\documentclass[11pt]{article}
\usepackage{blindtext}
\usepackage{enumitem} 
\setlist[enumerate]{parsep=0pt}
 
\usepackage{ragged2e} 

\usepackage{mathrsfs}
\usepackage{amssymb}
\usepackage{layout}
\usepackage{graphicx}
\usepackage{marvosym}
\usepackage[usenames,dvipsnames]{color}
\usepackage{verbatim}
\usepackage{enumitem}
\usepackage[pdftex]{hyperref}
\usepackage{fancyhdr}
\usepackage{indentfirst}
\usepackage{centernot}

\usepackage{amsthm,etoolbox}

\makeatletter
\@addtoreset{theorem}{section}% Reset theorem counter with every section
\@addtoreset{theorem}{subsection}
\@addtoreset{theorem}{subsubsection}

\usepackage{amsmath}

\usepackage{hyperref}% http://ctan.org/pkg/hyperref
\usepackage{cleveref}[noabbrev]% http://ctan.org/pkg/cleveref
\newtheorem{theorem}{Theorem}
\newtheorem{definition}[theorem]{Definition}
\newtheorem{lemma}[theorem]{Lemma}
\newtheorem{prop}[theorem]{Proposition}
\newtheorem{cor}[theorem]{Corollary}
\newtheorem{fact}[theorem]{Fact}
\newtheorem{remark}[theorem]{Remark}
\crefname{theorem}{Theorem}{Theorems}
\crefname{lemma}{Lemma}{Lemmas}
\crefname{prop}{Proposition}{Propositions}
\crefname{fact}{Fact}{Facts}
\crefname{remark}{Remark}{Remarks}
\crefname{cor}{Corollary}{Corollaries}

\def\tp{\mbox{\rm tp}}
\def\qftp{\mbox{\rm qftp}}
\def\cb{\mbox{\rm cb}}
\def\wcb{\mbox{\rm wcb}}

\newcommand{\theoremprefix}{}
\let\thetheoremsaved\thetheorem
\renewcommand{\thetheorem}{\theoremprefix\thetheoremsaved}

\patchcmd{\@startsection}{\par}{\renewcommand{\theoremprefix}{\csname the#1\endcsname.}}{}{}

\makeatother

\begin{document}

\pagenumbering{arabic}

\newcommand{\forkindep}[1][]{%
  \mathrel{
    \mathop{
      \vcenter{
        \hbox{\oalign{\noalign{\kern-.3ex}\hfil$\vert$\hfil\cr
              \noalign{\kern-.7ex}
              $\smile$\cr\noalign{\kern-.3ex}}}
      }
    }\displaylimits_{#1}
  }
}
\newpage
\begin{center}

    \LARGE \MakeUppercase{Kim-forking for hyperimaginaries in NSOP1 Theories}

\vspace{5mm}

    \Large Yvon \textsc{Bossut}
\end{center}
\vspace{10mm}

\begin{center}
\justify
Abstract : We adapt the properties of Kim-independence in NSOP1 theories with existence proven in \cite{kaplan2020kim},\cite{dobrowolski2022independence} and \cite{chernikov2023transitivitylownessranksnsop1} by Ramsey, Kaplan, Chernikov, Dobrowolski and Kim to hyperimaginaries by adding the assumption of existence for hyperimaginaries. We show that Kim-independence over hyperimaginaries satisfies a version of Kim's lemma, symmetry, the independence theorem, transitivity and witnessing. As applications we adapt Kim's results around colinearity and weak canonical bases from \cite{kim2021weak} to hyperimaginaries and get some results using boundedly closed hyperimaginaries and amalgamation for Kim-forking over them.\footnote{Partially supported by ANR-DFG AAPG2019 GeoMod}
\end{center}

\vspace{20pt}

\paragraph{{\Large Introduction}}

\justify
NSOP1 theories have recently been studied as a generalisation of simple theories. Kim and Pillay \cite{kim1997simple} have shown the basic properties of forking in simple theories, in particular symmetry, transitivity and amalgamation, and they have shown that simplicity can be characterised in terms of the existence of an independence notion satisfying these properties (and some more which are obvious for forking independence). Byunghan Kim defined the notion of Kim-forking, a kind of generic forking independence, and Kaplan and Ramsey \cite{kaplan2020kim} proved their basic properties, symmetry, transitivity and amalgamation. Moreover, Chernikov and Ramsey \cite[Theorem 5.8]{chernikov2016model} showed that the existence of an independence relation that satisfies these properties characterise NSOP1 theories.

\justify
In this paper we generalise the known results of Kim-forking in NSOP1 theories (with existence) to hyperimaginaries. Hyperimaginaries constitute a necessary part of the study of forking in simple theories, in general canonical bases are only known to exist as hyperimaginaries (see \cite{casanovas2011simple}), but they also have an interest in more general contexts : for exemple the elements of the quotient of a group by a type definable subgroup only exist as hyperimaginaries.

\justify
We shall begin with a reminder about forking and hyperimaginaries and an introduction to NSOP1 theories. In the second section we shall study Kim-Forking for hyperimaginaries in NSOP1 theories with existence for hyperimaginaries. The results of this section are new for the most part. Lastly, in the third section, we shall generalise the results of \cite{kim2021weak} to hyperimaginaries. In that paper, Kim introduces the notion of weak canonical bases in the context of NSOP1 theories with existence, and proves some of their properties under the additional assumption of elimination of hyperimaginaries (see \cite{kim2014simplicity} for elimination of hyperimaginaries). He does so in order to only use Kim-Forking on real tuples and not on hyperimaginaries, where the properties of Kim-forking had not yet been developed. Our results on Kim-forking for hyperimaginaries will thus allow us to present these results in a more general context.

\newpage

\section{General facts}

\subsection{Forking and hyperimaginaries}

\justify
The definition and basic properties of forking and hyperimaginaries can be found in \cite{kim2014simplicity}, the following results are proved in \cite{casanovas2011simple}.

\begin{definition} Let $b_{0},e$ be hyperimaginaries and $p(x,b_{0})$ be a partial type over $b_{0}$.\begin{enumerate}
\item $p(x,b_{0})$ divides over $e$ if there is an $e$-indiscernible sequence $(b_{i}$ : $i < \omega)$ such that $ \lbrace p (x,b_{i})$ : $i<\omega \rbrace$ is inconsistent.
\item $p (x,b_{0})$ forks over $e$ if it implies a finite disjunction of formulas each of which divides over $e$.
\item Let $a \in \mathbb{M}^{heq}$. We write $a\forkindep_{e}b_{0}$ to denote the assertion that $\tp(a/eb_{0})$ does not fork over $e$ and $a\forkindep^{d}_{e}b_{0}$ to denote the assertion that $\tp(a/eb_{0})$ does not divide over $e$.
\item A sequence of hyperimaginaries $(a_{i}$ : $i\in I)$ is called $e$-Morley if it is $e$-indiscernible and if $a_{i}\forkindep_{e}a_{<i} $ for every $i \in I $.
\end{enumerate}
\end{definition}

\begin{definition} Let $p(x)\in S_{E}(e)$ for some hyperimaginary $e$ and let $A$ be a class of hyperimaginaries. We say that $p$ is finitely satisfiable in $A$ if for every formula $\varphi \in p$ (considering $p$ as a partial type over some representative of $e$, which we assume is closed under conjunctions) there is some $b_{E}\in A$ such that $\models \varphi(\hat{b})$ for every representative $\hat{b}$ of $b_{E}$. This definition is independent of the chosen representative of $e$. We will write $a\forkindep^{\mathcal{U}}_{e}b$ to mean that $\tp(a/eb)$ is finitely satisfiable in $dcl^{heq}(e)$.
\end{definition}

\begin{prop} Let $a,b,e \in \mathbb{M}^{heq}$. If $\tp(a/eb)$ is finitely satisfiable in $dcl^{heq}(e)$ then $\tp(a/eb)$ does not fork over $e$, i.e. $a\forkindep^{\mathcal{U}}_{e}b$ implies $a\forkindep_{e}b$.
\end{prop}

\begin{prop}\label{hypindisc} Let $e \in \mathbb{M}^{heq}$ and $(a_{i}$ : $i \in I)$ be a $e$-indiscernible sequence of hyperimaginaries. Then there is $\hat{e}$ and $(\hat{a}_{i}$ : $i \in I)$ such that $\hat{e}$ is a  representative of $e$, $\hat{a}_{i}$ is a representative of $a_{i}$ for every $i \in I$ and $(\hat{a}_{i}$ : $i \in I)$ is $\hat{e}$-indiscernible.
\end{prop}

\begin{lemma}\label{dividhyp} Characterisation of $e$-dividing : For any $a,b,e \in \mathbb{M}^{heq}$ the following are equivalent :
\begin{enumerate}
\item[(1)] $\tp (a/eb)$ does not divide over $e$.
\item[(2)] For every infinite $e$-indiscernible sequence I such that $b\in I$, there is some $a'\equiv_{eb}a$ such that $I$ is $ea'$-indiscernible.
\item[(3)] For every infinite $e$-indiscernible sequence $I$ such that $b\in I$, there is some $J\equiv_{eb}I$ such that $J$ is $ea$-indiscernible.
\end{enumerate}\end{lemma}

\begin{prop} Left transitivity of division : Let $a,b,c,e \in \mathbb{M}^{heq}$. If $\tp(b/ce)$ does not divide over $e$ and $\tp(a/cbe)$ does not divide over $be$ then $\tp(ab/ce)$ does not divide over $e$.
\end{prop}

\begin{prop} Extension for hyperimaginaries : \begin{enumerate}

\item Let $\pi (x)$ be a partial type over $A$. If $\pi (x)$ does not fork over $e \in \mathbb{M}^{heq}$ there is $p \in S(A)$ a completion of $\pi$ that does not fork over $e$.
\item Let $a,b,c \in \mathbb{M}^{heq}$ such that $a\forkindep_{b}c $. Then, for any $d \in \mathbb{M}^{heq}$ there is $a'\equiv_{bc}a$ such that $a'\forkindep_{b}cd $.
\end{enumerate}
\end{prop}

\begin{prop} Left transitivity of forking for hyperimaginaries (12.7 and 12.10 of \cite{casanovas2011simple}) :  Let $a,b,c,e \in \mathbb{M}^{heq}$. If $a\forkindep_{e}c$ and $b\forkindep_{ea}c$ then $ab\forkindep_{e}c$.
\end{prop}

\begin{prop} Let $a,b,e\in \mathbb{M}^{heq}$. Then $a\forkindep_{e}b$ if and only if $a\forkindep_{e}bdd(b)$ iff $a\forkindep_{bdd(e)}b$.
\end{prop}

\begin{remark}\label{bondedmorley} A sequence is $e$-Morley if and only if it is $bdd(e)$-Morley.
\end{remark}

\justify
The following is the hyperimaginary version of \cite[Lemma 2.2.5]{kim2014simplicity}, the proof carries over verbatim.

\begin{lemma}\label{morleylascdist} Let $e \in \mathbb{M}^{heq}$ and $I = (a_{i}$ : $i<\omega)$ be an $e$-Morley sequence. Suppose that for some $n < \omega$, $ (a^{j}_{0},...,a^{j}_{n}$ : $j<\omega)$ is an $e$-indiscernible sequence of $\tp(a_{0} ... a_{n}/e)$. Then there is
a sequence $J\equiv_{e}I$ such that $a^{j}_{0},...,a^{j}_{n}\frown J$ is an $e$-indiscernible
sequence for each $j < \omega$, so in particular by invariance $a^{j}_{0},...,a^{j}_{n}\frown J$ is also $e$-Morley.
\end{lemma}

\subsection{Existence for hyperimaginaries}

\begin{definition} A theory $T$ is said to have existence for hyperimaginaries if for every $a,e\in \mathbb{M}^{heq}$ we have $a\forkindep_{e}e$.
\end{definition}

\justify
We know that simple theories have existence for hyperimaginaries and most of the interest so far in hyperimaginaries has been in the context of simple theories : the existence of canonical bases, amalgamation of types. There is not much literature about this axiom.  We do not know if NSOP1 theories have existence, an other open question is : does existence imply existence for hyperimaginaries in NSOP1 theories?

\justify
The following is an improvement of \cite[Lemma 2.12]{dobrowolski2022independence}, the proof of the first point works verbatim. We assume that our theory $T$ has existence for hyperimaginaries.

\begin{lemma}\label{kfhyp} Let $e \in \mathbb{M}^{heq}$ and $I = ( a_{i}$ : $i<\omega )$ be an $e$-Morley sequence of hyperimaginaries that is $c$-indiscernible for $c\in \mathbb{M}^{heq}$ such that $e \in dcl^{heq}(c)$. Then for any $b  \in \mathbb{M}^{heq}$ there is $( b_{i}$ : $i<\omega )$ such that $b_{0}\equiv_{ea_{0}}b$ and $( a_{i}b_{i}$ : $i<\omega )$ is $e$-Morley and $c$-indiscernible.
\end{lemma}

\justify
\begin{proof} We begin by extending $I$ to a $c$-indiscernible sequence $( a_{i}$ : $i<\kappa )$. We have a sequence $(b'_{i}$ : $i<\kappa )$ such that $J' =( a_{i}b'_{i}$ : $i<\kappa )$ is $e$-Morley and $b'_{0}\equiv_{ea_{0}}b$. By Erdös-Rado we have a $c$-indiscernible sequence $\tilde{J}=( \tilde{a}_{i}\tilde{b}_{i}$ : $i<\omega )$ finitely based on $J'$ over $c$, so it is $e$-Morley. Since $I$ is $c$-indiscernible we have that $I\equiv_{c}( \tilde{a}_{i}$ : $i<\omega )$, so $( a_{i}b_{i}$ : $i<\omega )\equiv_{c}( \tilde{a}_{i} \tilde{b}_{i}$ : $i<\omega )$ for some new sequence $( b_{i}$ : $i<\omega )$. We then have that $(a_{i}b_{i}$ : $i<\omega )$ is $c$-indiscernible, $e$-Morley and that $b_{0}\equiv_{ea_{0}}b$.\end{proof}

\subsection{NSOP1 theories}

\justify
We give a proposition that will be useful for later and that gives us a more easily exploitable characterisation of NSOP1. This proposition is \cite[Proposition 2.4]{kaplan2020kim}.

\begin{definition} A formula $\varphi (x,y) $ in $\mathcal{L}$ has the strong order property of the first kind ($SOP_{1}$) in T if there is a tree of parameters $(a_{\eta})_{\eta \in 2^{<\omega}}$ in a model of T such that : \begin{center}
$\left\{ \begin{array}{l}
        \ \forall \eta \in 2^{\omega}$, $\lbrace \varphi (x,a_{\eta \lceil\alpha})$ : $\alpha < \omega \rbrace $ is consistent$  \\
        \ \forall \nu,\eta \in 2^{<\omega}$, if $\nu \unrhd \eta\frown 0  $ : $\lbrace \varphi (x,a_{\nu})$,  $\varphi (x,a_{\eta \frown 1}) \rbrace $ is inconsistent$
    \end{array}
    \right.  $ 
\end{center}
The theory T has $SOP_{1}$ if one of its formulas does, T is $NSOP_{1}$ otherwise.\end{definition}

\begin{prop}\label{syntax3} Let $T$ be a complete theory. The following are equivalent : \begin{enumerate}
\item[(1)] T has $SOP_{1}$.
\item[(2)] There is a formula $\varphi(x,y) $ and a sequence $(c_{i,j}$ : $i<\omega,j<2)$ such that : \begin{center}
$\left\{ \begin{array}{l}
        \ (a)$ $\forall i < \omega$, $c_{i,0}\equiv_{c_{<i}}c_{i,1}  \\
        \ (b)$ $\lbrace \varphi (x,c_{i,0})$ : $i <\omega \rbrace $ is consistent$ \\
        \ (c)$ $\lbrace \varphi (x,c_{i,1})$ : $i < \omega \rbrace $ is 2-inconsistent $
    \end{array}
    \right.  $
\end{center}
    \item[(3)] There is a formula $\varphi(x,y) $, an integer k and a sequence $(c_{i,j}$ : $i<\omega,j<2)$ such that :\begin{center}
     $\left\{ \begin{array}{l}
        \ (a)$ $\forall i < \omega$, $c_{i,0}\equiv_{c_{<i}}c_{i,1}  \\
        \ (b)$ $\lbrace \varphi (x,c_{i,0})$ : $i <\omega \rbrace $ is consistent$ \\
        \ (c)$ $\lbrace \varphi (x,c_{i,1})$ : $i < \omega \rbrace $ is k-inconsistent $
    \end{array}
    \right.  $
    \end{center}
\end{enumerate} \end{prop}

\section{Kim-Forking for hyperimaginaries}

\justify
We can find definitions about Kim-forking in \cite{dobrowolski2022independence}, we will here extend them to hyperimaginaries.

\justify
We will prove our results about Kim-forking for hyperimaginaries in the same order as in the articles \cite{dobrowolski2022independence} and \cite{chernikov2023transitivitylownessranksnsop1}, and we will mainly adapt proofs from these two articles. Adapting these proofs can be really direct if they only use types, but the proofs that uses formulas need more work. In the latter case we will often choose some representatives and work with real tuples. We will give all proofs that differs from the original ones. We shall work in a monster model $\mathbb{M}$ of an NSOP1 theory $T$ with existence for hyperimaginaries.

\subsection{Definitions and basic properties}

\begin{definition} Let $b_{0},e$ be hyperimaginaries and $p(x,b_{0})$ be a partial type over $b_{0}$.\begin{enumerate}

\item $p(x, b_{0})$ Kim-divides over $e$ if there is an $e$-Morley sequence $(b_{i}$ : $i<\omega )$ such that $\lbrace p(x,b_{i})$ : $i<\omega \rbrace$ is inconsistent. We will say that $p(x, b_{0})$ Kim-divides over $e$ with respect to $(b_{i}$ : $i<\omega)$.

\item $p(x,b_{0})$ Kim-forks over $e$ if it implies a finite disjunction of formulas each of which Kim-divides over $e$.

\item Let $a \in \mathbb{M}^{heq}$. We write $a\forkindep^{K}_{e}b_{0}$ to denote the assertion that $\tp(a/eb_{0})$ does not Kim-fork over $e$ and $a\forkindep^{Kd}_{e}b_{0}$ to denote the assertion that $\tp(a/eb_{0})$ does not Kim-divide over $e$.

\item $I=(a_{i}$ : $i<\omega)$ is Kim-Morley over $e$ if it is $e$-indiscernible and if $a_{i}\forkindep^{K}_{e}a_{<i}$ for every $i<\omega$.
\end{enumerate}
\end{definition}

\justify
As equivalent types implies the same formulas the second point does not depend on the choice of representatives. The definition we gave here is for hyperimaginaries but we can reformulate it in the sense of real tuples thanks to \cref{kfhyp} :

\begin{remark}\label{kimdivbasic}
Let $E$ be a type definable equivalence relation, $b_{0}$ be an $E$-class and $p(x,b_{0})$ be partial type over $b_{0}$. Then $p(x,b_{0})$ Kim-divides over $e$ if and only if there is a representative $\overline{b}_{0}$ such that $p(x,\overline{b}_{0})$ Kim-divides over $e$, seeing $p$ as a partial type over $\overline{b}_{0}$. So a partial type over $b$ Kim-divides over $e$ if and only if it implies a formula that Kim-divides over $e$.
\end{remark}

\begin{remark} Since Kim-division implies division we have $a\forkindep_{e}b$ implies $a\forkindep^{K}_{e}b$ for $a,b,e\in \mathbb{M}^{heq}$.
\end{remark}

\justify
The following result and it's proof are analogous to the case of usual forking \cite[Lemma 2.1.6]{kim2014simplicity} :

\begin{lemma}\label{kimdivh} Characterisation of Kim-dividing for hyperimaginaries : Let $a,e,b \in \mathbb{M}^{heq}$. TFAE: 
\begin{enumerate}
\item[(1)] $\tp(a/eb)$ does not Kim-divide over $e$.
\item[(2)] For every infinite $e$-Morley sequence $I$ such that $b\in I$, there is some $a'\equiv_{eb}a$ such that $I$ is $ea'$-indiscernible.
\item[(3)] For every infinite $e$-Morley sequence $I$ such that $b\in I$, there is some $J\equiv_{eb}I$ such that $J$ is $ea$-indiscernible.

\end{enumerate}\end{lemma}

\subsection{Kim's Lemma for Kim-forking over hyperimaginaries}

\justify
These proofs are adaptations of those from the section 3 of \cite{dobrowolski2022independence}. We shall reduce to the case of Kim-forking over models and then use the results proven in \cite{kaplan2020kim}. We begin by introducing a more general notion of indiscernibility for trees.

\begin{definition} For an ordinal $\alpha $ let the language $\mathcal{L}_{s,\alpha}$ be $\langle \unlhd ,\wedge ,<_{lex},(P_{\beta})_{\beta < \alpha} \rangle $. We may view a tree with $\alpha $ levels as an $\mathcal{L}_{s,\alpha}$-structure by interpreting $\unlhd$ as the tree partial order, $\wedge$ as the meet function, $<_{lex}$ as the lexicographical order and $P_{\beta}$ to define level $\beta$.\end{definition}

\begin{definition} Suppose $I$ is an $\mathcal{L}'$ structure, where $\mathcal{L}'$ is some language and let $e=\overline{e}_{E}\in \mathbb{M}^{heq}$.
\begin{enumerate}
\item We say that $(a_{i}$ : $i \in I)$ is an $I$-indexed indiscernible set (eventually hyperimaginary) over $e$ if whenever $(s_{0},...,s_{n-1}),(t_{0},...,t_{n-1})$ are tuples from $I$ with 
$\qftp_{\mathcal{L}'}(s_{0},...,s_{n-1})=\qftp_{\mathcal{L}'}(t_{0}, ...,t_{n-1})$, we have $\tp(a_{s_{0}}, ...,a_{s_{n-1}}/e)=\tp(a_{t_{0}},...,a_{t_{n-1}}/e)$
\item In the case that $\mathcal{L}'= \mathcal{L}_{s,\alpha}$ for some $\alpha$ and $I$ is a tree, we say that an $I$-indexed indiscernible set is an s-indiscernible tree. Since the only $\mathcal{L}_{s,\alpha}$ structure we consider will be trees, we will often refer $I$-indexed indiscernible sets in this case as s-indiscernible trees.
\item We say that $I$-indexed indiscernible sets have the modeling property if, given any $I$-indexed set of real tuples $(a_{i}$ : $i \in I)$ and any set of real parameters $A$, there is an $I$-indexed indiscernible set $(b_{i}$ : $i \in I)$ over $A$ locally based on $(a_{i}$ : $i \in I)$ over $A$, meaning that for any finite set of $\mathcal{L}_{A}$-formulas $\Delta$ and finite tuples $(t_{0}, ...,t_{n-1})$ from $I$, there is a tuple $(s_{0}, ...,s_{n-1})$ from $I$ such that $\qftp_{\mathcal{L}'}(s_{0},...,s_{n-1})=\qftp_{\mathcal{L}'}(t_{0}, ...,t_{n-1})$, and also $\tp_{\Delta}(b_{s_{0}},...,b_{s_{n-1}})=\tp_{\Delta}(a_{t_{0}}, ...,a_{t_{n-1}})$.
\end{enumerate}\end{definition}

\begin{fact}\label{modprop0} \cite[Theorem 4.3]{kim2014tree} : $\omega^{<\omega}$ indexed trees have the modeling property. \end{fact}

\justify
We will now describe the notion that we will use in the context of hyperimaginaries as the equivalent of what locally based is in \cite{dobrowolski2022independence} and \cite{chernikov2023transitivitylownessranksnsop1}, it will be useful to help us preserve type definable properties when inductively building trees :

\begin{definition}
If $(a_{i}$ : $i \in I)$, $(b_{i}$ : $i \in I)$ are $I$-indexed hyperimaginaries and $e=(\overline{e})_{E} \in \mathbb{M}^{heq}$, we say that $(b_{i}$ : $i \in I)$ satisfies the EM-type of $(a_{i}$ : $i \in I)$ over $e$ if :

\justify
When $q(y_{0},..,y_{n-1})$ is a quantifier free type in $\mathcal{L}'$, $E_{i}$ is a $\emptyset$-type definable equivalence relation for $i<n$ and $p(\overline{x},e)=\exists z \overline{x}'(E(z,\overline{e})\wedge \bigwedge \limits_{i< n}^{} E_{i}(x'_{i},x_{i})\wedge \pi(\overline{x}',z))$ is a partial $(E_{0},..,E_{n-1})$-type over $e$ such that for every tuple of indices $s_{0}, ...,s_{n-1} \in I$ such that $\models q(\overline{s})$ we have that $a_{s_{i}}$ is of sort of $E_{i}$ and $\models p(\overline{a}_{s_{0}},...,\overline{a}_{s_{n-1}},\overline{e})$ for some (any) representatives, then the same holds in $(b_{i}$ : $i \in I)$ : for every tuple of indices $s_{0}, ...,s_{n-1} \in I$ such that $\models q(\overline{s})$ we have that $b_{s_{i}}$ is of sort of $E_{i}$ and $\models p(\overline{b}_{s_{0}},...,\overline{b}_{s_{n-1}},\overline{e})$ for some (any) representatives. We will write this $(b_{i}$ : $i \in I) \models EM((a_{i}$ : $i \in I)/e)$.\end{definition}

\begin{lemma}\label{modprophyperimag}
If $(a_{\eta})_{\eta \in \omega^{<\omega}}$ is a tree of hyperimaginaries, with $a_{\eta}$ of sort $E_{i}$ if $dom(\eta)=[0,i)$, and if $e\in \mathbb{M}^{heq}$ then there is an $\omega^{<\omega}$-indexed indiscernible tree over $e$ that satisfies $EM((a_{\eta})_{\eta \in \omega^{<\omega}}/e)$ (meaning that we can complete the EM-type).
\end{lemma}

\justify
\begin{proof}
We consider a tree of representatives $(\overline{a}\eta)_{\eta \in \omega^{<\omega}}$ and a representative $\overline{e}$ of $e$. By \cref{modprop0} there is a tree $(\overline{b}\eta)_{\eta \in \omega^{<\omega}}$ finitely based over $\overline{e}$ on $(\overline{a}\eta)_{\eta \in \omega^{<\omega}}$ and indiscernible over $\overline{e}$, now set $b_{\eta}= (\overline{b}_{\eta})_{E_{i}}$ if $dom(\eta)=[0,i)$. This is an indiscernible tree over $e$ and it satisfies $EM((a_{\eta})_{\eta \in \omega^{<\omega}}/e)$.
\end{proof}

\justify
The following lemma is an adaptation of \cite[Lemma 2.16]{dobrowolski2022independence}, it is the same proof but we need to give a meaning of $Cl$ for hyperimaginaries. Let $e = (\overline{e})_{E}\in \mathbb{M}^{heq}$. For $M$ a model such that $e \in dcl^{heq}(M)$ and a finite real tuple $a$ let : $Cl_{e}(a/M):=\lbrace \exists z(E(z,\overline{e})\wedge \varphi(x,y,z))$ : $\models \exists z(E(z,\overline{e})\wedge \varphi(a,m,z))$ for some $m\in M$ and $\varphi(x,y,z)$ an $\mathcal{L}$-formula$\rbrace$.

\begin{lemma}\label{hyp6} Let $I =(a_{i}$ : $i<\omega)$ be an $e$-Morley sequence of real tuples. Then there is a model $M$ such that $e\in dcl^{heq}(M)$, $M \forkindep_{e}I$, $I$ is $M$-indiscernible, and $\tp(a_{<k}/M a_{\geq k})$ is an heir extension of $\tp(a_{<k}/M)$ for any $k < \omega$ (so $I$ is a coheir sequence over $M$).
\end{lemma}

\justify
\begin{proof}
We consider the class of models :
\begin{center}
$\mathcal{U}_{0}:=\lbrace N\prec \mathbb{M}$ : $e \in dcl^{heq}(N),$ $\vert N \vert < \overline{\kappa}$, $I$ is $N$-indiscernible and $N \forkindep_{e}I\rbrace$.
\end{center}We order it by $N_{1}<N_{1}$ if $N_{1}\prec N_{2}$ and $Cl_{e}(I/N_{1}) \varsubsetneq Cl_{e}(I/N_{2})$.
\justify
By compactness we can lengthen $I$ to a still $e$-Morley sequence $\tilde{I}=(a_{i}$ : $i<\kappa)$.
By existence there is a model $M'$ that contains a representative of $e$ and such that $M'\forkindep_{e}\tilde{I}$. We can then extract an $M'$-indiscernible sequence  $I'$ locally based on $\tilde{I}$ over $M'$, in particular $I'$ is $e$-indiscernible, $I\equiv_{e}I'$ and $M'\forkindep_{e}I'$. By sending $I'$ to $I$ over $e$ we find a model $N$ such that $I$ is $N$-indiscernible, $e \in dcl^{heq}(N)$ and $N\forkindep_{e}I$. So $\mathcal{U}_{0}$ is a non empty inductive order, so it has a maximal element $M_{0}$ by Zorn's Lemma.

\justify
\textbf{Claim :} For any $k<l<\omega$, $\varphi(\overline{x},y,z)$ $\mathcal{L}$-formula and $m \in M_{0}$ such that $\models \exists z(E(z,\overline{e})\wedge\varphi(a_{<k},m,a_{k},...a_{l},z))$ there is $m'\overline{m}''\in M_{0}$ such that $\models \exists z(E(z,\overline{e})\wedge\varphi(a_{<k},m',\overline{m}'',z))$.

\justify
From this point the proof carries over similarly as the original one.\end{proof}

\justify
The following lemma is an adaptation of \cite[Lemma 3.4]{dobrowolski2022independence}, the proof carries over verbatim.

\begin{lemma}\label{hyp4} Suppose $e \in \mathbb{M}^{heq}$, $I = (a_{i}$ : $i<\omega)$ and $J =  (b_{i}$ : $i<\omega)$
are $e$-indiscernible sequences of real elements with $a_{0} = b_{0}$ and $b_{>0} \forkindep^{d}_{e}b_{0}$. Then there is a tree $(c_{\eta})_{\eta \in \omega^{< \omega}}$ satisfying the following properties:

\begin{enumerate}
\item[(1)] For all $\eta \in \omega^{<\omega}$, $(c_{\eta \frown \langle i\rangle})_{i < \omega}\equiv_{e}I $.
\item[(2)] For all $\eta \in \omega^{<\omega}$, $(c_{\eta},c_{\eta \lceil l(\eta) -1 },...,c_{\eta \lceil 0}) \equiv_{e}(b_{0},...,b_{l(\eta)}) $.
\item[(3)] $(c_{\eta})_{\eta \in \omega^{< \omega}}$ is s-indiscernible over $e$.
\end{enumerate}
\end{lemma}

\justify
We now prove the main result of this subsection, for the purpose of which we recall : 

\begin{definition} Let $e \in \mathbb{M}^{heq}$, $(a_{i,j})_{i<\kappa , j<\lambda}$ is a mutually indiscernible array over $e$ if, for each $i < \kappa$, the sequence $\overline{a}_{i} = (a_{i,j}$ : $j < \lambda)$ is indiscernible over $e\overline{a}_{\neq i}$.
\end{definition}

\justify
Here we use \cref{hyp6} to reduce to the case of Kim-forking over models, for which results have already been proven in \cite{kaplan2020kim}. This proof is a simplified version of \cite[Theorem 3.5]{dobrowolski2022independence}.

\begin{theorem}\label{kimlemmahyp} $T$ satisfies Kim’s lemma for Kim-dividing over an hyperimaginary $e$ : a formula $\varphi(x, a)$ Kim-divides over $e$ with respect to some Morley sequence in $\tp(a/e)$ if and only if it Kim-divides over $e$ with respect to any such sequence.
\end{theorem}

\begin{proof}

\justify
Towards contradiction, assume we are given $\varphi(x, a)$ and $e$-Morley sequences $I = (a_{i}$ : $i <\omega)$ and $J = (b_{i}$ : $i <\omega)$ both of which are in $\tp(a/e)$, and such that $\lbrace\varphi(x, a_{i})$  : $ i < \omega\rbrace$ is consistent and $\lbrace\varphi(x, b_{i})$  : $ i < \omega\rbrace$ is inconsistent.

\justify
As $J$ is a Morley sequence over $e$, we have $b_{>0} \forkindep^{d}_{e}b_{0}$ so by \cref{hyp4} there is a tree $(c_{\eta})_{\eta \in \omega^{< \omega}}$ satisfying :

\begin{enumerate}
\item[(1)] For all $\eta \in \omega^{<\omega}$, $(c_{\eta \frown \langle i\rangle})_{i < \omega}\equiv_{e}I $.
\item[(2)] For all $\eta \in \omega^{<\omega}$, $(c_{\eta},c_{\eta \lceil l(\eta) -1 },...,c_{\eta \lceil 0}) \equiv_{e}(b_{0},...,b_{l(\eta)}) $.
\item[(3)] $(c_{\eta})_{\eta \in \omega^{< \omega}}$ is s-indiscernible over $e$.
\end{enumerate}

\justify
Define an array $(d_{i,j} )_{i,j < \omega}$ by $d_{i,j} = c_{0^{i}\frown \langle j \rangle}$. By s-indiscernibility, $(d_{i,j} )_{i,j < \omega}$ is mutually indiscernible over $e$ and, moreover, $\overline{d}_{i}\equiv_{e}I $ for all $i < \omega$ by (1). By compactness applied to $(d_{i,j} )_{i,j < \omega}$, for $\kappa$ large enough, we can find an $e$-mutually indiscernible array $(h_{i,j} )_{i< \kappa,j < \omega}$ such that $(h_{i,0} $ : $i< \kappa)$ has the same EM-type over $e$ as $J$ with the reversed order - by (2) - so in particular ($\dagger$) : $\lbrace\varphi(x, h_{i,0})$  : $ i < \omega\rbrace$ is inconsistent and $\overline{h}_{i} \equiv_{e} I$ (by (1)) for all $i<\kappa$. By Erdös-Rado, we may assume $\kappa = \omega$ and $(\overline{h}_{i}$ : $i<\omega)$ is $e$-indiscernible.
 
\justify
By applying \cref{hyp6}, we can find a model $M$ such that $e \in dcl^{heq}(M) $, $\overline{h}_{0}$ is a Morley sequence over $M$ and $M \forkindep_{e} \overline{h}_{0} $. Then because $(\overline{h}_{i}$ : $i<\omega)$ is $e$-indiscernible, we may,
moreover, assume that $M$ has been chosen so that $(\overline{h}_{i}$ : $i<\omega)$ is $M$-indiscernible by \cref{kimdivh}. Let $\lambda$ be any cardinal larger than $2^{\vert \mathcal{L} \vert + \vert M \vert}$ and apply compactness to stretch the array to $(h_{i,j} )_{i < \omega , j < \lambda}$, preserving $e$-mutual indiscernibility, the $M$-indiscernibility of $(\overline{h}_{i}$ : $i<\omega)$ and the EM-type of $(h_{i,0}$ : $i< \omega)$ over $e$.

\justify
Now by induction on $n < \omega$, we find $\alpha_{n} < \lambda$ so that $h_{n,\alpha_{n}}\forkindep^{K}_{M}h_{<n,\alpha_{<n}}$ for all $n < \omega$ :

\justify
Suppose we have found $(\alpha_{m}$ : $m < n)$. Then by the pigeonhole principle
and the choice of $\lambda$, there is an infinite subsequence $I_{n}$ of $\overline{h}_{n}$ so that every tuple of $I_{n}$ has the same type over $Mh_{<n,\alpha_{<n}}$. Let $\alpha_{n} < \lambda$ be least such that $h_{n,\alpha_{n}} \in I_{n}$.
As $I_{n}$ is a subsequence of a Morley sequence over $M$, $I_{n}$ is also Morley over $M$ and $Mh_{<n,\alpha_{<n}}$-indiscernible hence, by Kim’s lemma for Kim-dividing (over models), we have $h_{<n,\alpha_{<n}}\forkindep^{K}_{M}I_{n}$, so $h_{<n,\alpha_{<n}}\forkindep^{K}_{M}h_{n,\alpha_{n}}$ which, by symmetry over models, is what we need to complete the induction.

\justify

We claim that $\lbrace \varphi (x, h_{n,\alpha_{n}})$ : $n < \omega\rbrace$ is consistent. By compactness, it suffices to show that $\lbrace \varphi (x, h_{n,\alpha_{n}})$ : $n < N \rbrace$ does not Kim-divide over $M$ for any $N$. This is true for $N = 1$ by Kim’s lemma, since $\lbrace \varphi (x, h_{0,j})$ : $j < \lambda\rbrace$ is consistent and $\overline{h}_{0}$ is a Morley sequence over $M$. Assuming we have shown it for $N$, we can choose $c_{N} \models \lbrace \varphi (x, h_{n,\alpha_{n}})$ : $n < N \rbrace$ with $c_{N}\forkindep^{K}_{M}h_{<N,\alpha_{<N}}$. Additionally,
since $h_{0\alpha_{0}} \equiv_{M} h_{N,\alpha_{N}}$, we can choose $c$ so that $c_{N} h_{0,\alpha_{0}} \equiv _{M}ch_{N,\alpha_{N}}$, from which follows that $c\forkindep^{K}_{M}h_{N,\alpha_{N}}$.

\justify
By applying the independence theorem over
$M$ (\cite[Theorem 6.5]{kaplan2020kim}), we find :\begin{center}
$c_{N+1} \models \tp(c_{N} /Mh_{<N,\alpha_{<N}} ) \cup \tp(c/Mh_{N,\alpha_{N}} )$ such that $c_{N+1}\forkindep^{K}_{M}Mh_{\leq N,\alpha_{\leq N}}$.
\end{center} In particular, we have $c_{N+1} \models \lbrace \varphi (x, h_{n,\alpha_{n}})$ : $n < N +1 \rbrace$, and therefore $\lbrace \varphi (x, h_{n,\alpha_{n}})$ : $n < N+1 \rbrace$ does not Kim-divide over $M$, completing the induction.

\justify
By mutual indiscernibility over $e$, we have that $\lbrace \varphi (x, h_{i,0})$ : $i < \omega \rbrace$ is consistent, contradicting ($\dagger$).\end{proof}

\begin{prop}\label{kimfkimd} Kim-forking = Kim-dividing : For any $e\in \mathbb{M}^{heq}$, $\varphi(x,y)$ $\mathcal{L}$-formula and $b$ real tuple if $\varphi(x, b)$ Kim-forks over $e$ then $\varphi(x, b)$ Kim-divides over $e$.
\end{prop}

\justify

\begin{proof} Suppose $\varphi (x,b) \vdash \bigvee\limits_{j< n }^{}  \psi_{j}(x,c^{j})$, where each $\psi_{j}(x,c^{j}) $ Kim-divides over $e$. Let $(b_{i}, c_{i}^{0},...,c^{n-1}_{i}$ : $i< \omega)$ be a Morley sequence in $\tp(b_{i}, c^{0},...,c^{n-1}/e)$. Since $(b_{i}$ : $i< \omega)$ is a Morley sequence in $\tp(b/e)$, to get that $\varphi(x, b)$ Kim-divides over $e$ it is enough to show that $\lbrace \varphi(x, b_{i})$ : $i < \omega\rbrace$ is inconsistent.

\justify
If not, then there is some $a \models \lbrace \varphi(x, b_{i})$ : $i < \omega\rbrace$. Then, by the pigeonhole principle, we get that, for some
$j < n$, $a$ realizes $\psi_{j} (x, c_{i}^{j})$, for infinitely many $i$’s. As $(c_{i}^{j}$ : $i< \omega)$ is a Morley sequence in $\tp(c^{j}/e)$, it follows from Kim’s Lemma for Kim-forking (\cref{kimlemmahyp}) that $\psi_{j} (x, c^{j})$ does not Kim-divide over $e$, a contradiction.\end{proof}

\begin{lemma}\label{kimlemmahyp2+} Let $I = ( a_{i}$ : $i<\omega )$ be $e$-Morley and $ea'$-indiscernible, then $a'\forkindep^{K}_{e}a_{0}$.
\end{lemma}

\begin{proof}
\justify
We will reduce to the case of $I$ being a sequence of real tuples : By \cref{kfhyp} there is $( \hat{e}'_{i}\hat{a}'_{i}$ : $i<\omega ) $ such that $\hat{a}'_{i}$ is a representative of $a_{i}$, $\hat{e}'_{i}$ is a representative of $e$ and $(a_{i}\hat{e}'_{i}\hat{a}'_{i}$ : $i<\omega ) $ is $e$-Morley and  $ea'$-indiscernible.

\justify
Let $p(x,\hat{e}_{0},\hat{a}_{0}) = \tp(a'/\hat{e}_{0}\hat{a}_{0})$. Then by indiscernibility $a' \models \lbrace p(x,\hat{e}_{i},\hat{a}_{i})$ : $i<\omega \rbrace$, so by \cref{kimlemmahyp} $p(x,\hat{e}_{0},\hat{a}_{0})$ does not Kim-divide over $e$. By \cref{kimfkimd} it does not Kim-fork over $e$, so $a'\forkindep^{K}_{e}\hat{e}_{0},\hat{a}_{0}$, so $a'\forkindep^{K}_{e}a_{0}$.\end{proof} 

\begin{cor}\label{kimlemmahyp2} Kim's Lemma for Kim-forking over hyperimaginaries : If $I = ( a_{i}$ : $i<\omega )$ is $e$-Morley and $ea'$-indiscernible then $a'\forkindep^{K}_{e}I$.
\end{cor}

\begin{proof}\justify
It is enough to show that $a'\forkindep_{e}a_{[0,n)}$ for every $n < \omega $, but this is a direct consequence of \cref{kimlemmahyp2+} and of the fact that $(a_{[n\cdot i,n \cdot (i +1))}$ : $i<\omega)$ is also an $e$-Morley sequence.
\end{proof}

\begin{remark}\label{bondedkimmorley} Since a sequence is $a$-indiscernible if and only if it is $bdd(a)$-indiscernible for any $a$ we have $a\forkindep^{K}_{e}b$ if and only if $a\forkindep^{K}_{bdd(e)}b$ if and only if $bdd(a)\forkindep^{K}_{e}b$ 
\end{remark}

\begin{prop}\label{kimhyp1} Let $a,b,e \in \mathbb{M}^{heq}$. If $a\forkindep^{K}_{e}b$ and $I = (b_{i}$ : $i< \omega)$ is a Morley sequence in $\tp(b/e)$ with $b_{0} = b$, then there is some $a'\equiv_{eb} a$ such that $a'\forkindep^{K}_{e}I$ and $I$ is $ea'$-indiscernible
\end{prop}

\begin{proof}

\justify
 We have that $a\forkindep^{Kd}_{e}b$, so by \cref{kimdivh} there is $a'\equiv_{eb} a$ such that $I$ is $ea'$-indiscernible, then by \cref{kimlemmahyp2} we have $a'\forkindep^{K}_{e}I$.\end{proof}

\justify
The following is a new result, we will use it a lot to reduce to the case of real tuples.

\begin{lemma}\label{liftkimmorley} Let $a,b,e \in \mathbb{M}^{heq}$ such that $a\forkindep^{K}_{e}b$. There are $\overline{a},\overline{b}$ representatives of $a$ and $b$ respectively such that $\overline{a}\forkindep^{K}_{e}\overline{b}$.
\end{lemma}

\begin{proof}

\justify
Let $\hat{b}$ be a representative of $b$, by extension there is $a'\equiv_{eb}a$ such that $a'\forkindep^{K}_{e}\hat{b}b$, and by sending $a'$ to $a$ over $eb$ we find an other representative of $b$, $\overline{b}$, such that $a\forkindep^{K}_{e}\overline{b}$. Now by \cref{kimhyp1} there is $I=(b_{i}$ : $i<\omega)$ an $e$-Morley sequence that is $ea$-indiscernible. By \cref{hypindisc} there is $\overline{a}$ a representative of $a$ such that this sequence is $e\overline{a}$-indiscernible. By Kim's Lemma for Kim-forking (\cref{kimlemmahyp2}) we have that $\overline{a}\forkindep^{K}_{e}I$, so $\overline{a}\forkindep^{K}_{e}\overline{b}$.\end{proof}

\subsection{Symmetry of Kim-Forking for hyperimaginaries}

\justify
Proofs in this section are for the most part rewritings of those from the section 4 of \cite{dobrowolski2022independence}. We will begin by introducing the key tools for this section and the following ones, tree Morley sequences :

\begin{definition} Suppose $\alpha$ is an ordinal. Let $[\alpha ]:=\lbrace \beta <\alpha$ : $ \beta $ is not limit$\rbrace$. We define $\mathcal{T}_{\alpha}$ to be the set of functions $f$ such that :
\begin{enumerate}
\item[•] $dom(f)$ is an end segment $[\beta ,\alpha )$ for $\beta \in [\alpha ]$.
\item[•] $ran(f)\subseteq \omega$.
\item[•] $f$ has a finite support, meaning that the set $ \lbrace \gamma \in dom(f)$ : $f(\gamma )\neq \emptyset \rbrace $ is finite.
\end{enumerate}\end{definition}
\justify
We interpret $\mathcal{T}_{\alpha}$  as an $\mathcal{L}_{s,\alpha}$-structure by defining : 
\begin{enumerate}
\item[•] $f \unlhd g $ iff $f\subseteq g$. If $\neg (f \unlhd g) $ and $\neg (g \unlhd f) $ we write $f\perp g $.
\item[•] $f\wedge g = f\vert_{[\beta , \alpha )} = g\vert_{[\beta , \alpha )} $ where $ \beta = $min$\lbrace \gamma$ $: $ $f\vert_{[\gamma , \alpha )} = g\vert_{[\gamma , \alpha )}\rbrace $, it is in $\mathcal{T}_{\alpha}$ by finite support.
\item[•] $f<_{lex}g$ iff $f\triangleleft g $ or  $f\perp g $ with $dom(f\wedge g)=[\gamma ,\alpha )$ and $f(\gamma ) < g(\gamma )$.
\item[•] For all $\beta \leq \alpha $, $P_{\beta}= \lbrace f\in \mathcal{T}_{\alpha} $ : $dom(f)=[\beta,\alpha )\rbrace$.
\end{enumerate}

\begin{definition} Suppose that $\alpha $ is an ordinal. 
\begin{enumerate}
\item Restriction : If $w \subseteq \alpha $, the restriction of $\mathcal{T}_{\alpha}$ to the set of levels $w$ is given by : \begin{center}
$\mathcal{T}_{\alpha}\lceil w=\lbrace \eta \in \mathcal{T}_{\alpha}$ : $min(dom(\eta))\in w $ and $supp(\eta )\subseteq w \rbrace $,
\end{center} $supp$ being the support of the functions.
\item Concatenation : If $\eta\in \mathcal{T}_{\alpha} $, $dom(\eta )=[\beta +1, \alpha )$, and $i<\omega $, let $\eta\frown \langle i \rangle $ denote the function $\eta \cup \lbrace (\beta ,i)\rbrace $. We define $\langle i \rangle \frown \eta \in \mathcal{T}_{\alpha+1}$ to be $\eta \cup \lbrace (\alpha ,i)\rbrace $. We write $\langle i \rangle $ for $\langle i \rangle \frown  \emptyset $.
\item Canonical inclusions : If $\alpha < \beta $, we define the map $\iota_{\alpha \beta}$ : $\mathcal{T}_{\alpha}\rightarrow \mathcal{T}_{\beta}$ by $\iota_{\alpha \beta}(f)=f\cup \lbrace (\gamma , \emptyset)$ : $\gamma \in \beta \setminus \alpha \rbrace $.
\item The all 0's path : If $\beta <\alpha $, let $\zeta_{\beta} $ denote the function $\lbrace (\gamma ,0)$ : $\gamma \in [\beta ,\alpha) \rbrace $, if $\beta \in [\alpha ]$ it is an element of $\mathcal{T}_{\alpha}$ and it is the $<_{lex}$-smallest element of $P_{\beta}$. 
\end{enumerate}\end{definition}

\begin{definition} Suppose that $(a_{\eta})_{\eta \in \mathcal{T}_{\alpha}}$ is a tree of hyperimaginaries and $e$ an hyperimaginary.
\begin{enumerate}
\item We say  $(a_{\eta})_{\eta \in \mathcal{T}_{\alpha}}$ is weakly spread out over $e$ if for all $\eta \in \mathcal{T}_{\alpha}$ with $P_{\beta +1}(\eta)$ for some $\beta < \alpha $, $(a_{\unrhd \eta\frown \langle i\rangle}$ : $i<\omega)$ is a Morley sequence in $\tp(a_{\unrhd \eta\frown \langle 0\rangle}/e)$.
\item Suppose $(a_{\eta})_{\eta \in \mathcal{T}_{\alpha}}$ is a tree which is weakly spread out and s-indiscernible over $e$, and that for all $v,w \in [\alpha ]^{<\omega}$ with $\vert v \vert = \vert w \vert $,
$(a_{\eta})_{\eta \in \mathcal{T}_{\alpha}\lceil v} \equiv_{e} (a_{\eta})_{\eta \in \mathcal{T}_{\alpha}\lceil w}$, we then say that $(a_{\eta})_{\eta \in \mathcal{T}_{\alpha}}$ is a weakly Morley tree over $e$.

\item A weak tree Morley sequence over $e$ is a $e$-indiscernible sequence of the form $(a_{\zeta_{\beta}})_{\beta \in [\alpha ]}$ for some weakly Morley tree $(a_{\eta})_{\eta \in \mathcal{T}_{\alpha}}$ over $e$.

\end{enumerate}\end{definition}

\justify
In an s-indiscernible tree over $e$ any two paths have the same type over $e$. Hence, (3) may be stated as : a weak tree Morley sequence over $e$ is a path in some weakly Morley tree over $e$.

\begin{fact}\label{modprop}\cite[Theorem 5.6]{kaplan2020kim} : $\mathcal{T}_{\alpha}$-indexed indiscernibles have the modeling property for any $\alpha $. By the same argument as in \cref{modprophyperimag} if $(a_{\eta})_{\eta \in \mathcal{T}_{\alpha}}$ is a tree of hyperimaginaries with $a_{\eta}$ of sort $E_{\beta}$ if $dom(\eta)=[\beta,\alpha)$ and $e\in \mathbb{M}^{heq}$ there is $(b_{\eta})_{\eta \in \mathcal{T}_{\alpha}}\models EM((a_{\eta})_{\eta \in \mathcal{T}_{\alpha}}/e)$ that is s-indiscernible over $e$.\end{fact}

\justify
The following lemma will help us build weak Morley trees - it is actually our only reliable provider, it is proved over models (\cite[Lemma 5.10]{kaplan2020kim}). The proof for the case of hyperimaginaries works verbatim.

\begin{lemma}\label{morleytree0} Suppose that $(a_{\eta})_{\eta \in \mathcal{T}_{\kappa}}$ is a tree of hyperimaginaries, weakly spread out over $e$ and s-indiscernible over $e \in \mathbb{M}^{heq}$. If $\kappa$ is sufficiently large, then there is a weak Morley tree over $e$ $(b_{\eta})_{\eta \in \mathcal{T}_{\omega}}$ such that for all $v \in \omega^{<\omega}$, there is $v' \in [\kappa]^{<\omega}$ so that : $$(a_{\eta})_{\eta \in \mathcal{T}_{\kappa}\lceil v} \equiv_{e} (b_{\eta})_{\eta \in \mathcal{T}_{\omega}\lceil v'}.$$
\end{lemma}

\justify
The following result is the analogous of \cref{kfhyp} for weak tree Morley sequence, it is original and holds in any theory with existence for hyperimaginaries :

\begin{lemma}\label{morleytree1} Let $b,e \in \mathbb{M}^{heq}$ and $(a_{i}$ : $i<\omega)$ be a weak tree Morley sequence of hyperimaginaries over $e$. Then there is $(b_{i}$ : $i<\omega)$ such that $(a_{i}b_{i}$ : $i<\omega)$ is a weak tree Morley sequence over $e$ such that $a_{i}b_{i} \equiv_{e} ba_{0} $ for every $i<\omega$.
\end{lemma} 
 
\justify\begin{proof}
Let $(a_{\eta})_{\eta \in \mathcal{T}_{\omega}}$ be a weak Morley tree over $e$ such that $a_{\zeta_{i}} = a_{i}$ for every $i<\omega$. By compactness we can stretch it to a tree $(a_{\eta})_{\eta \in \mathcal{T}_{\kappa}}$ such that if $v \in [\kappa ]^{<\omega}$ then for all $w \in [\omega ]^{<\omega}$ such that $\vert v \vert = \vert w \vert $ we have $(a_{\eta})_{\eta \in \mathcal{T}_{\kappa}\lceil v} \equiv_{e} (a_{\eta})_{\eta \in \mathcal{T}_{\omega}\lceil w}$. Then this tree is also weakly Morley over $e$.

\justify
We build by induction on $\alpha \leq \kappa$ a tree $S^{\alpha} = (a'^{\alpha}_{\eta}b'^{\alpha}_{\eta})_{\eta \in \mathcal{T}_{\alpha}}$ weakly spread out over $e$ and $s$-indiscernible over $e$ such that : 
\begin{enumerate}
\item[(1)] $(a'^{\alpha}_{\eta})_{\eta \in \mathcal{T}_{\alpha}}\equiv_{e}(a_{\iota_{\alpha,\kappa}(\eta)})_{\eta \in \mathcal{T}_{\alpha}}$.
\item[(2)] $a'^{\alpha}_{\eta}b'^{\alpha}_{\eta}\equiv_{e}a_{0}b$ for all $\eta \in \mathcal{T}_{\alpha}$.
\item[(3)] If $\alpha < \beta < \kappa$ we have $a'^{\alpha}_{\eta}= a'^{\beta}_{\iota_{\alpha, \beta}(\eta)}$ for all $\eta \in \mathcal{T}_{\alpha}$ and $(a'^{\alpha}_{\eta}b'^{\alpha}_{\eta})_{\eta \in \mathcal{T}_{\alpha}}\equiv_{e}(a'^{\beta}_{\iota_{\alpha,\beta}(\eta)}b'^{\beta}_{\iota_{\alpha,\beta}(\eta)})_{\eta \in \mathcal{T}_{\alpha}}$.
\end{enumerate}

\justify
For $\alpha = 0$ we take $a'^{0}_{\emptyset}b'^{0}_{\emptyset}\equiv_{e}a_{0}b$. For the limit step we take $a'^{\alpha}_{\eta}= a'^{\beta}_{\nu}$ for any $\beta < \alpha$ and $\nu \in \mathcal{T}_{\beta }$ such that $\iota_{\beta , \alpha}(\nu) = \eta $, by $(3)$ this is well defined, then by compactness we can find $(b'^{\alpha}_{\eta})_{\eta \in \mathcal{T}_{\alpha}}$ that satisfies the requirements : we take the limit type of the trees previously built.

\justify
Assume that $S^{\alpha}$ has been built for $\alpha$ a limit ordinal. There is $c'$ such that $c'(a'^{\alpha}_{\eta})_{\eta \in \mathcal{T}_{\alpha}}\equiv_{e}a_{\iota_{\alpha+1,\kappa}(\emptyset)}(a_{\iota_{\alpha,\kappa}(\eta)})_{\eta \in \mathcal{T}_{\alpha}}$. Then we also have a $d'$ such that $c'd'\equiv_{e}a_{0}b$. Now let $(c_{\eta}d_{\eta})_{\eta \in \mathcal{T}_{\alpha+1}}$ be an s-indiscernible tree over $e$ that satisfies $EM((c'd')S^{\alpha}/e)$ by \cref{modprop}. Then we have $(c_{\eta})_{\eta \in \mathcal{T}_{\alpha+1}}\equiv_{e}(a_{\iota_{\alpha +1,\kappa}(\eta)})_{\eta \in \mathcal{T}_{\alpha +1}}$, and we take $S^{\alpha+1}$ to be the image of $(c_{\eta}d_{\eta})_{\eta \in \mathcal{T}_{\alpha+1}}$ under this $e$-automorphism.

\justify
Assume that $S^{\alpha}$ has been built for $\alpha$ a successor ordinal. We can find $a'_{\emptyset}\langle (a'_{\langle i \rangle\frown \eta})_{\eta \in \mathcal{T}_{\alpha}}$ : $i<\omega \rangle$ such that $a'_{\emptyset}\langle (a'_{\langle i \rangle\frown \eta})_{\eta \in \mathcal{T}_{\alpha}}$ : $i<\omega \rangle \equiv_{e} (a_{\iota_{\alpha +1,\kappa}(\eta)})_{\eta \in \mathcal{T}_{\alpha+1}}$ and $(a'_{\langle 0 \rangle\frown \eta})_{\eta \in \mathcal{T}_{\alpha}} = (a'^{\alpha}_{\eta})_{\eta \in \mathcal{T}_{\alpha}}$. We then find $b'_{\emptyset}$ such that $a'_{\emptyset}b'_{\emptyset}\equiv_{e}a_{0}b$. Then $\langle (a'_{\langle i \rangle\frown \eta})_{\eta \in \mathcal{T}_{\alpha}}$ : $i<\omega \rangle$ is an $e$-Morley sequence that is $ea'_{\emptyset}$-indiscernible, we extend it to $\langle (a'_{\langle i \rangle\frown \eta})_{\eta \in \mathcal{T}_{\alpha}}$ : $i<\kappa \rangle$ also $ea'_{\emptyset}$-indiscernible. Set $a'^{\alpha+1}_{\emptyset}=a'_{\emptyset}$ and $a'^{\alpha+1}_{\langle i \rangle \frown \eta} = a'_{\langle i \rangle\frown \eta}$.

\justify
We apply \cref{kfhyp} to find $\langle (b'_{\langle i \rangle\frown \eta})_{\eta \in \mathcal{T}_{\alpha}}$ : $i<\kappa \rangle$ such that $(a'_{\langle 0 \rangle\frown \eta})_{\eta \in \mathcal{T}_{\alpha}}(b'_{\langle 0 \rangle\frown \eta})_{\eta \in \mathcal{T}_{\alpha}}\equiv_{e} (a'^{\alpha}_{\eta})_{\eta \in \mathcal{T}_{\alpha}}(b'^{\alpha}_{\eta})_{\eta \in \mathcal{T}_{\alpha}}$ and $\langle (a'_{\langle i \rangle\frown \eta})_{\eta \in \mathcal{T}_{\alpha}}(b'_{\langle i \rangle\frown \eta})_{\eta \in \mathcal{T}_{\alpha}}$ : $i<\kappa \rangle$ is $e$-Morley and $ea'_{\emptyset}$-indiscernible. Then by Erdös-Rado we can extract an $ea'_{\emptyset}b'_{\emptyset}$-indiscernible sequence $\langle (a''_{\langle i \rangle\frown \eta})_{\eta \in \mathcal{T}_{\alpha}}(b''_{\langle i \rangle\frown \eta})_{\eta \in \mathcal{T}_{\alpha}}$ : $i<\omega \rangle$, by sending $\langle (a''_{\langle i \rangle\frown \eta})_{\eta \in \mathcal{T}_{\alpha}}$ : $i<\omega \rangle$ to $\langle (a'_{\langle i \rangle\frown \eta})_{\eta \in \mathcal{T}_{\alpha}}$ : $i<\omega \rangle$ over $ea'_{\emptyset}$ we find $b'^{\alpha+1}_{\emptyset}$ and $\langle (b'^{\alpha+1}_{\langle i \rangle\frown \eta})_{\eta \in \mathcal{T}_{\alpha}}$ : $i<\omega \rangle$ such that $\langle (a'^{\alpha+1}_{\langle i \rangle\frown \eta})_{\eta \in \mathcal{T}_{\alpha}}(b'^{\alpha+1}_{\langle i \rangle\frown \eta})_{\eta \in \mathcal{T}_{\alpha}}$ : $i<\kappa \rangle$ is $e$-Morley and $ea'^{\alpha+1}_{\emptyset}b'^{\alpha+1}_{\emptyset}$-indiscernible. This gives us the desired $S^{\alpha+1}$ and concludes our induction.

\justify
Now by considering $S^{\kappa}$ for a large enough $\kappa$ and applying \cref{morleytree0} we find a weak tree Morley sequence $(a'_{i}b'_{i}$ : $i<\omega)$ such that $(a'_{i}$ : $i<\omega)\equiv_{e}(a_{i}$ : $i<\omega)$ and $a'_{0}b'_{0}\equiv_{e}a_{0}b$. Sending $(a'_{i}$ : $i<\omega)$ to $(a_{i}$ : $i<\omega)$ over $e$ gives us the desired sequence.\end{proof}

\justify
The following lemma is an adaptation of \cite[Lemma 4.6]{dobrowolski2022independence}, the proof carries over verbatim.

\begin{lemma}\label{morleytreehyp} Let $a,b,e \in \mathbb{M}^{heq}$. If $a\forkindep^{K}_{e}b$, then for any ordinal $\alpha \geq 1$, there is a tree $(c_{\eta})_{\eta \in \mathcal{T}_{\alpha}}$ weakly spread out over $e$ and s-indiscernible over $e$ such that if $\eta \triangleleft \nu$ and $dom(\nu) = \alpha$, then
$c_{\eta}c_{\nu}\equiv_{e}ab$.\end{lemma}

\justify
The following lemma is an adaptation of \cite[Lemma 4.7]{dobrowolski2022independence} :

\begin{lemma}\label{morleytree2} Let $a,b,e \in \mathbb{M}^{heq}$. If $a \forkindep^{K}_{e} b$, then there is a weak tree
Morley sequence $(a_{i}$ : $i<\omega)$ over $e$ which is $eb$-indiscernible with $a_{0} = a$.
\end{lemma}

\begin{proof}
\justify
By \cref{liftkimmorley} we have $\overline{a} \forkindep^{K}_{e} \overline{b}$ for some representatives. Clearly, proving the lemma for these real tuples is enough as we can then take the sequence of the $E$-classes for the right type equivalence relation $E$. The proof in this case carries over verbatim using \cref{morleytreehyp} and \cref{morleytree0}.\end{proof}

\begin{prop} Suppose that $(a_{i}$ : $i<\omega)$ is a weak tree Morley sequence of real tuples over $e \in \mathbb{M}^{heq}$. Then $\lbrace \varphi(x, a_{i})$ : $i < \omega \rbrace$ is inconsistent if and only if
$\varphi(x, a_{0})$ Kim-divides over $e$.
\end{prop}

\begin{proof}
\justify
Let $(a_{\eta})_{\eta \in \mathcal{T}_{\omega}}$ be a weak Morley tree over $e$ with $a_{i}=a_{\zeta_{i}}$. Let $\eta_{i} \in \mathcal{T}_{\omega}$ be the function with $dom(\eta_{i}) = [i, \omega)$ and :

\begin{center}
$ \eta_{i}(j) = $ $\left\{ \begin{array}{l}
        \ 1$ if $j=i\\
        \ 0$ otherwise.$
    \end{array}
    \right. $ 
\end{center}

\justify
We consider the sequence $I = (a_{\eta_{i}},a_{\zeta_{i}}$ : $i<\omega)$. Because $(a_{\eta})_{\eta \in \mathcal{T}_{\omega}}$ is a weak Morley tree over $e$,
$I$ is an $e$-indiscernible sequence. Moreover, by s-indiscernibility, $a_{\eta_{i}}\equiv_{eI_{>i}}a_{\zeta_{i}}$ for all $i<\omega$.

\justify
By \cref{syntax3} and $NSOP_{1}$, it follows that  $\lbrace \varphi(x,a_{\eta_{i}})$ : $i<\omega\rbrace $ is consistent if and only if $\lbrace \varphi(x,a_{\zeta_{i}})$ : $i<\omega\rbrace $ is consistent. Because $(a_{\eta})_{\eta \in \mathcal{T}_{\omega}}$ is a spread out tree over $e$ $(a_{\unrhd \zeta_{i}\frown \langle k \rangle}$ : $ k<\omega)$ is $e$-Morley. So $a_{\unrhd \zeta_{i}\frown \langle 1 \rangle}\forkindep_{e}a_{\unrhd \zeta_{i}\frown \langle 0\rangle} $ and since $a_{\eta_{i}} \in a_{\unrhd \zeta_{i}\frown \langle 1 \rangle} $ and $a_{\eta_{<i}} \subseteq a_{\unrhd \zeta_{i}\frown \langle 0\rangle} $ $a_{\eta_{i}}\forkindep_{e}a_{\eta_{<i}} $ for all $i<\omega$, meaning that $(a_{\eta_{i}}$ : $i<\omega)$ is $e$-Morley. 

\justify
By Kim's Lemma for Kim-dividing we have that $\varphi (x,a_{0})$ Kim-divides over $e$ if and only if $\varphi (x,a_{\eta_{0}})$ Kim-divides over $e$ if and only if $\lbrace \varphi(x,a_{\eta_{i}})$ : $i<\omega\rbrace $ is inconsistent if and only if $\lbrace \varphi(x,a_{i})$ : $i<\omega\rbrace $ is inconsistent.\end{proof}

\justify
From this we deduce that weak Morley tree sequences of real tuples witness Kim-Forking :

\begin{cor}\label{kimlemmatree} (Kim’s lemma for weak tree Morley sequences) Let $e \in \mathbb{M}^{heq}$. TFAE :
\begin{enumerate}
\item $\varphi (x,a_{0})$ Kim-divides over $e$.
\item For some weak tree Morley sequence $(a_{i}$ : $i<\omega)$ over $e$ $\lbrace \varphi(x,a_{i})$ : $i<\omega\rbrace $ is inconsistent.
\item For any weak tree Morley sequence $(a_{i}$ : $i<\omega)$ over $e$ $\lbrace \varphi(x,a_{i})$ : $i<\omega\rbrace $ is inconsistent.
\end{enumerate}\end{cor}

\justify
As in \cref{kimdivbasic} thanks to \cref{morleytree1} we can deduce the following hyperimaginary version :

\begin{cor}\label{kimlemmatreeversionhyp}
Let $b\in \mathbb{M}^{heq}$ and $p(x,b_{0})$ be a partial type over $b_{0}$. Then $p(x,b_{0})$ Kim-divides over $e$ iff there is an weak tree Morley sequence $(b_{i}$ : $i < \omega)$ of hyperimaginaries such that $\lbrace p(x,b_{i})$ : $i<\omega \rbrace$ is inconsistent iff for any weak tree Morley sequence $(b_{i}$ : $i < \omega)$ of hyperimaginaries $\lbrace p(x,b_{i})$ : $i<\omega \rbrace$ is inconsistent.
\end{cor}

\begin{cor}\label{remarkmorleytree} If $a,b_{0},e\in \mathbb{M}^{heq}$ and if $I=(b_{i}$ : $i<\omega )$ is a weak tree sequence that is $ea$-indiscernible then $a\forkindep^{K}_{e}b_{0}$.
\end{cor}

\begin{cor}\label{conjmorleytreebis} Let $a,e \in \mathbb{M}^{heq}$ and $I=(b_{i}$ : $i<\omega)$ be a weak tree Morley sequence of hyperimaginaries. Then if $a\forkindep^{K}_{e}b_{0}$ there is $a'\equiv_{eb_{0}}a$ such that $I$ is a $ea'$-indiscernible.
\end{cor}
\begin{proof}
It is similar to the one of $(1)\Rightarrow (2)$ of \cref{kimdivh}, by \cref{kimlemmatreeversionhyp} $\lbrace p(x,e,b_{i})$ : $i<\omega\rbrace$ is consistent for $p(x,e,b_{0})=\tp(a/eb_{0})$.\end{proof}

\begin{theorem}\label{symh} Kim-forking satisfies symmetry for hyperimaginaries.
\end{theorem}

\begin{proof}
\justify
Assume that $a\forkindep^{K}_{e}b$. Let $p(x,\overline{e}, \overline{a}) := \tp(b/ea)$. Since $a\forkindep^{K}_{e}b$ by \cref{morleytree2} there is a weak tree Morley sequence $(a_{i}$ : $i<\omega)$ over $e$ with $a_{0} = a$ which is $eb$-indiscernible. Then $\models p(\overline{b},\overline{e}, \overline{a}_{i})$ for all $i<\omega$ and any representatives, so by \cref{kimlemmatreeversionhyp} $\tp(b/ea)$ does not Kim-divide over $e$.\end{proof}

\justify
The following is \cite[Lemma 5.9]{kaplan2020kim}, it is going to be useful to exploit the fact that weak tree Morley sequences witness Kim-forking, the proof carries over verbatim.

\begin{lemma}\label{treemorleysequ} Suppose $(a_{i}$ : $i<\omega)$ is a weak tree Morley sequence of hyperimaginaries over $e\in \mathbb{M}^{heq}$.
\begin{enumerate}
\item[(1)] If $a_{i} \sim (b_{i}, c_{i})$ for all $i < \omega$, where the $b_{i}$ are all hyperimaginaries of the same type, then $(b_{i}$ : $i<\omega)$ is a weak tree Morley sequence over $e$.
\item[(2)] Given $1 \leq n < \omega$, let $d_{i}: = a_{[n\cdot i,n\cdot (i+1))}$. Then $(d_{i}$ : $i<\omega)$ is a weak tree Morley sequence over $e$.
\end{enumerate}
\end{lemma}

\justify
From this, finite character, \cref{remarkmorleytree} and \cref{morleytree1} we can deduce :
\begin{cor}\label{remarkmorleytree2bis} If $e\in \mathbb{M}^{heq}$ and $I=(b_{i}$ : $i<\omega)$ is a weak tree Morley sequence of hyperimaginaries that is $ea$-indiscernible for $a\in \mathbb{M}^{heq}$, then $a\forkindep^{K}_{e}I$.
\end{cor}

\justify
Now to put a bit of order between the many notions of sequences we have :

\begin{prop}\label{totalmorleyh} If $I=(a_{i}$ : $i<\omega)$ is weak tree Morley sequence over $e$ then it is a total Kim-Morley sequence over $e$, meaning that $a_{ > i}\forkindep^{K}_{e}a_{\leq i}$ for all $i<\omega$.\end{prop}

\justify
\begin{proof}Suppose $I = ( a_{i} : i < \omega ) $ is a weak tree Morley sequence
over $e$. Let $(a_{\eta})_{\eta \in \mathcal{T}_{\alpha}}$ be a weak Morley tree over $e$ such that $a_{i}=a_{ \zeta_{i}} $ for all $i < \omega $. Then for all $i<\omega $, we have that $(a_{\unrhd \zeta_{i} \frown \langle j\rangle}$ : $j<\omega)$ is a Morley sequence in $\tp(a_{\unrhd \zeta_{i}\frown \langle 0\rangle}/e)$ and is $ea_{ \zeta_{ \geq  i +1}}$-indiscernible. Therefore, by Kim's Lemma for Kim-forking, $a_{ \zeta_{ \geq  i +1}}\forkindep^{K}_{e}a_{\unrhd \zeta_{i+1} \frown \langle 0\rangle}$, which implies $a_{ > i}\forkindep^{K}_{e}a_{\leq i}$ for all $i< \omega $.\end{proof} 

\begin{definition} A sequence $I = ( a_{i} : i < \omega )$ is called strong Kim-Morley over $e \in \mathbb{M}^{heq}$ if for every $n < \omega $ the sequence of the $\overline{a}_{i}:= (a_{i\cdot n},...,a_{i\cdot n+n-1})$ is Kim-Morley over $e$.\end{definition}

\begin{prop} $I = ( a_{i} : i < \omega ) $ is strong Kim-Morley over $e$ iff it is total Kim-Morley over $e$.\end{prop}

\justify
\begin{proof}

$[\Leftarrow ]$ : Immediate by finite character. $[\Rightarrow ]$ : Let $i < \omega $, we show that $a_{ \geq i}\forkindep^{K}_{e}a_{< i}$ for all $i< \omega $.
By finite character, it is enough to show that for all $i < j < \omega $, $a_{[i,j]}\forkindep^{K}_{e}a_{< i}$. For this let $k > j$, by assumption we have that $a_{[k,2k-1]}\forkindep^{K}_{e}a_{[0,k-1]}$, so by indiscernibility and invariance, $a_{[i,j]}\forkindep^{K}_{e}a_{< i}$.\end{proof}

\subsection{Hyperdefinability of Kim-Forking}

\justify
The following lemma is the adaptation of \cite[Lemma 1.18]{chernikov2023transitivitylownessranksnsop1}, it is a new result.

\begin{lemma}\label{typedef1} Let $e\in \mathbb{M}^{heq}$ of sort $E$ and $p \in S_{F}(e)$ a complete $F$-type over $e$.
\begin{enumerate}
\item[(1)]  Given any type definable equivalence relation $E'$, the set of hyperimaginaries $(a,b)$ such that $b\models p$, $a$ of sort $H$ and $a\forkindep^{K}_{e}b$ is $e$-hyperdefinable.
\item[(2)] The set of Kim-Morley sequences in $p$ is $e$-hyperdefinable.
\end{enumerate}
\end{lemma}

\justify
\begin{proof}
(1) : Lets first assume that $a$ is a real tuple.

\justify
Let $c\models p$, $\overline{c},\overline{e}$ be representatives of $c$ and $e$, and : \begin{center}
$\Gamma (x,y,e) := p(y) \cup\lbrace \exists y'z (F(y',y) \wedge E (z,e) \wedge ( \neg \exists y''z'( \varphi (x,y'',z')\wedge \theta(y'',y')\wedge \rho(z',z))))$ : $\neg \exists y''z'( \varphi (x,y'',z')\wedge \theta(y'',\overline{c}')\wedge \rho(z',\overline{e}'))$ Kim-divides over $e$ for any representatives $\overline{c}',\overline{e}'$ of $c$ and $e\rbrace$.\end{center}

\justify
By symmetry, invariance, and Kim-forking = Kim-dividing, if $b\equiv_{e}c$ we have that $a\centernot\forkindep^{K}_{e}b$ if and only if $a \models \exists y''z'( \varphi (x,y'',z')\wedge \theta'(y'',\overline{b})\wedge \rho'(z',\overline{e}))$ with $\overline{b},\overline{e}$ representatives of $b$ and $e$, $\theta'\in F$, $\rho'\in E$ and $\exists y''z'( \varphi (x,y'',z')\wedge \theta'(y'',\overline{b})\wedge \rho'(z',\overline{e}))$ Kim divides over $e$. Now if $\theta \in F$ is such that $\theta^{2}\vdash\theta'$ and $\rho \in E$ is such that $\rho^{2}\vdash\rho'$ we have that $a \models \exists y''z'( \varphi (x,y'',z')\wedge \theta(y'',\overline{b}')\wedge \rho(z',\overline{e}'))$ and $\exists y''z'( \varphi (x,y'',z')\wedge \theta(y'',\overline{b}')\wedge \rho(z',\overline{e}'))$ Kim divides over $e$ for any representatives  $\overline{b}',\overline{e}'$ of $b$ and $e$. So $\exists y''z'( \varphi (x,y'',z')\wedge \theta(y'',\overline{c}')\wedge \rho(z',\overline{e}'))$ Kim divides over $e$ for any representatives  $\overline{c}',\overline{e}'$ of $c$ and $e$.

\justify
From this we deduce that this partial type is as desired. Now for the case of $a$ hyperimaginary by \cref{liftkimmorley} we have that $a$ is Kim-independent with $b$ over $e$ if and only if some of its representatives are. So if $\Gamma (x,y,\overline{e})$ is as previously we just take the type $\exists x' (E'(x',x)\wedge \Gamma (x',y,\overline{e}))$.

\justify
(2) :  One can take $\Delta$ to be the partial type that asserts $(x_{i}$ : $i<\omega)$ is $e$-indiscernible, every $x_{i}\models p$, and $x_{i}\forkindep^{K}_{e}x_{<i}$ for every $i<\omega$, by (1) it defines a partial type over $e$.\end{proof}

\subsection{Amalgamation for hyperimaginaries}

\justify
The content of this section consists in a rewriting of the proofs of section 5 of \cite{dobrowolski2022independence}.

\begin{remark}\label{weaktrans} We do have the following results :
\begin{enumerate}
\item[•] Symmetry : For all $a,b,e \in \mathbb{M}^{heq}$, $a\forkindep^{K}_{e}b$ if and only if $b\forkindep^{K}_{e}a$
\item[•] Some weak version of transitivity : if $a\forkindep^{d}_{e}bc$ and $b\forkindep^{K}_{e}c$ then $ab\forkindep^{K}_{e}c$
\end{enumerate}

\end{remark}

\begin{remark}\label{lascd} For any $a,b,e \in \mathbb{M}^{heq}$ we have that $a \equiv_{e}^{L}b$ if and only if the Lascar distance $d_{e}(a,b)$ is finite. This is equivalent to the "Morley"-Lascar distance $Md_{e}(a,b)$ being finite, where $Md_{e}(a,b)$ is defined in the same way as $d_{e}(a,b)$, except $e$-indiscernible sequences involved in the distance definition are all $e$-Morley sequences.
\end{remark}

\begin{proof}
\justify
It is enough to show that for an $e$-indiscernible sequence $I$ containing $a_{0}, a_{1}$, there is an $e$-Morley $J$ such that $a_{0}\frown J$ and $a_{1}\frown J$ are $e$-Morley, and this follows from \cref{morleylascdist}.\end{proof}

\justify
The following corollary is going to be useful, in the context of hyperimaginaries we can often replace hyperimaginaries with their bounded closure and assume that types are Lascar strong types :

\begin{remark} By \cref{kimhyp1} we have that $a\forkindep^{K}_{e}b$ if and only if there is an $ea$-indiscernible $e$-Morley sequence containing $b$. A sequence is $c$-indiscernible if and only if it is $bdd(c)$-indiscernible for any $c \in \mathbb{M}^{heq}$. Using symmetry, $a\forkindep^{K}_{e}b$ iff $bdd(a)\forkindep^{K}_{e}bdd(b)$. By using extension on the type over the bounded closure, we have an extension of types over the bounded closure (and we already know that NSOP1e theories are G-compact from \cite[Corollary 5.9]{dobrowolski2022independence}).
\end{remark}

\begin{cor}\label{exthyp} Let $b,c,e \in \mathbb{M}^{heq}$. Assume $b\forkindep^{K}_{e}c$. Then for any $a$ there is $h \equiv_{be}^{L}a$ such that $hb\forkindep^{K}_{e}c$. Moreover, for any $d$ there is $b'\equiv^{L}_{ec}b$ such that $b'\forkindep^{K}_{e}cd$. \end{cor}

\begin{lemma}\label{5.3} Let $a,b,c,e \in \mathbb{M}^{heq}$. Assume that $a\forkindep^{K}_{e}b$ and $a\forkindep^{K}_{e}c$. Then there is $h\in \mathbb{M}^{heq}$ such that $ac \equiv_{e} ah$, $b\forkindep_{e} h$ and $a\forkindep^{K}_{e} bh$.

\end{lemma}

\begin{proof}
\justify
We begin by establishing the following :

\justify
\textbf{Claim :} There is $c'$ such that $ac'\equiv_{e} ac$ and $a\forkindep^{K}_{e}bc'$.

\justify
Proof of the claim : Due to symmetry, we have $c\forkindep^{K}_{e}a$ and $b\forkindep^{K}_{e}a$. Hence there is a $e$-Morley sequence $J = (a_{i}$ : $i<\kappa)$ with $a_{0} = a$ which is $eb$-indiscernible by \cref{kimhyp1} for a large enough $\kappa$. By \cref{kimdivh} there is $c'' \equiv_{ea} c$ such that $J$ is $ec''$-indiscernible. By Erdös-Rado there is $J'= (a'_{i}$ : $i<\omega)$ a $ebc''$-indiscernible sequence finitely based on $J$ over $ebc''$, then we have $J'\equiv_{eb}J_{\omega}(:=(a_{i}$ : $i<\omega))$, and also $a_{i}c'' \equiv_{e} a'_{j}c$ for any $i,j<\omega$. Now if $f$ is some $eb$-automorphism sending $J'$ to $J_{\omega}$ then $J_{\omega}$ is $ac'$-indiscernible where $c'= f(c'')$, and $ac \equiv_{e}a'_{0}c'' \equiv_{e} ac'$. Hence  $bc'\forkindep^{K}_{e}a$, so $a\forkindep^{K}_{e}bc'$.

\justify

Choose $c'$ as in the claim and take a $e$-Morley sequence $I = (b_{i}c_{i}$ : $i < \omega)$ with
$b_{0}c_{0} = bc'$. By \cref{kimhyp1}, we can assume that $I$ is $ea$-indiscernible and $a\forkindep^{K}_{e}I$. Let $h = c_{1}$. Then $a\forkindep^{K}_{e}bh$, and since I is a Morley sequence over $e$, we have $b\forkindep_{e}h$.\end{proof}

\begin{lemma}\label{5.4} Let $a\forkindep^{K}_{e}b$ and $a\forkindep^{K}_{e}c$. Then there is $h$ such that $ac \equiv^{L}_{e} ah$ and $b\forkindep_{e}h$ and $a\forkindep^{K}_{e}bh$.
\end{lemma}

\justify
\begin{proof} By extension and existence for $\forkindep$, there is a model $M$ such that $M\forkindep_{e} abc$ and $e \in dcl^{heq}(M)$.
Then by \cref{weaktrans}, we have $Ma\forkindep^{K}_{e}b$ and $Ma\forkindep^{K}_{e}c$. Now we apply \cref{5.3} to $Ma, b$ and $c$ to obtain $h$ such that $Mah \equiv_{e} Mac$, $b\forkindep_{e}h$ and $Ma\forkindep^{K}_{e}bh$. So $ah \equiv_{M} ac$ and $ah \equiv^{L}_{e} ac$ follows.\end{proof}

\justify
The following result is the adaptation of \cite[Lemma 5.5]{dobrowolski2022independence}, the proof we give here is similar to the original one but is slightly more detailed.

\begin{lemma}\label{5.5} Zig-zag lemma : Let $b\forkindep^{K}_{e}c_{0}c_{1}$ and suppose there is an $e$-Morley sequence $I = (c_{i}$ : $i<\omega)$. Then there is a weak tree Morley sequence $(b_{i}c_{i}$ : $i<\omega)$ such that we have $b_{i}c_{i}\equiv_{e}bc_{0}$ for all $i<\omega$ and $b_{i}c_{j}\equiv_{e}bc_{1}$ for all $j>i$.

\end{lemma}

\justify
\begin{proof} Let $\kappa$ be an infinite cardinal. We begin by proving the following :

\justify
\textbf{Claim :} There is a $e$-Morley sequence $J = (d_{i}$ : $i<\kappa)$ of the same EM-type as $I$ over $e$ with $d_{0}=c_{0}$ such that $b\forkindep^{K}_{e}J$, $bd_{1}\equiv_{e}bc_{1}$ and $d_{>0}$ is $ebc_{0}$-indiscernible (so $bd_{i}\equiv_{e}bd_{1}\equiv_{e}bc_{1}$ for all $i>0$).

\justify
Proof of the claim : Let $I_{2} = (c_{2i}c_{2i+1}$ : $i<\omega)$. $I_{2}$ is $e$-Morley, and since $b\forkindep^{K}_{e}c_{0}c_{1}$ by \cref{kimhyp1} we can assume that $I_{2}$ is $eb$-indiscernible and $b\forkindep^{K}_{e}I_{2}$. 

\justify
We can extend $I_{2}$ to $I_{2}' = (\tilde{c}_{i}\tilde{c}'_{i}$ : $i<\kappa)$ still $eb$-indiscernible for a large enough $\kappa$. Let $J'=(\tilde{c}'_{i}$ : $i<\kappa)$, by Erdös-Rado we can extract $J_{1} = (c'_{2i+1}$ : $i<\omega)$ $ebc_{0}$-indiscernible and finitely based on $J'$ over $ebc_{0}$, and then extend $J_{1}$ to $J'_{1}$ an also $ebc_{0}$-indiscernible and finitely based on $J'$ over $ebc_{0}$ sequence of lenght $\kappa$. Then let $J:=c_{0}\frown J'_{1}$. $J'_{1}\equiv_{ec_{0}}I'_{>0}$, so $J=c_{0}\frown J'_{1}\equiv_{e} I$, and $J$ satisfies our conditions.

\justify
Now let $h_{0}=b$, we have that $d_{\geq 1}$ is an $e$-Morley sequence and is $eh_{0}d_{0}$-indiscernible, so $h_{0}d_{0}\forkindep^{K}_{e}d_{\geq 1}$. By extension there is $h_{1}$ such that $h_{1}d_{\geq 1}\equiv_{e}h_{0}J$ and $h_{0}d_{0}\forkindep^{K}_{e}h_{1}d_{\geq 1}$ $(\dagger)$. By \cref{kimhyp1} and symmetry there is $L_{1}=(h^{i}_{1}d^{i}_{1}$ : $i<\omega)$ $e$-Morley sequence with $h^{0}_{1}d^{0}_{1}=h_{0}d_{0}$ that is $eh_{1}d_{\geq 1}$-indiscernible, we define $d'_{1}=d_{1}$.

\justify
By the Standard Lemma there is $J_{2}$ such that $J_{2}\equiv_{ed_{0}d_{1}}d_{\geq 2}$ and $J_{2}$ is $eL_{1}h_{1}d_{1}$-indiscernible (we extend the EM-type of $d_{\geq 2}$ over $eL_{0}h_{1}d_{1}$). In particular $J_{2}$ is $e$-Morley and so $L_{1}h_{1}d_{1}\forkindep^{K}_{e}J_{2}$. Again by extension (as in $\dagger$) there is $h_{2}$ such that $h_{2}J_{2 }\equiv_{e} h_{1}d_{\geq 1}(\equiv_{e}h_{0}J)$ and $L_{1}h_{1}d_{1}\forkindep^{K}_{e}h_{2}J_{2}$. Hence there is a $e$-Morley sequence $L_{2}=(L^{i}_{1}h^{i}_{1}d^{i}_{1}$ : $i<\omega)$ such that $L^{0}_{1}h^{0}_{1}d^{0}_{1}=L_{1}h_{1}d'_{1}$ and that $L_{1}$ is $eh_{2}J_{2}$-indiscernible, let $d'_{2}$ be the first component of $J_{2}$.

\justify
Again by extension and the Standard Lemma there is  $h_{3}$ and $J_{3}$ $L_{2}h_{2}d'_{2}$-indiscernible such that $d_{0}d_{1}d'_{2}J_{3 }\equiv_{e} J$ and $L_{2}h_{2}d'_{2}\forkindep^{K}_{e}h_{3}J_{3}$ Then there is a $e$-Morley sequence $L_{3}=(L^{i}_{2}h^{i}_{2}d^{i}_{2}$ : $i<\omega)$ such that $L^{0}_{2}h^{0}_{2}d^{0}_{2}=L_{2}h_{2}d'_{2}$ and that $L_{3}$ is $eh_{3}J_{3}$-indiscernible.

\justify
We repeat this process to inductively construct $L_{\alpha}d'_{\alpha}h_{\alpha}J_{\alpha}$ for every $\alpha < \kappa$ such that $d'_{\alpha}$ is the first term of $J_{\alpha}$, $d'_{<\alpha}\frown J_{\alpha}\equiv_{e}J$, $h_{\alpha}J_{\alpha}\equiv_{e}h_{0}J$, $J_{\alpha}$ is $L_{<\alpha}h_{<\alpha}d'_{<\alpha}$-indiscernible and for $\alpha < \kappa$ limit $L_{\alpha} = \bigcup\limits_{\beta < \kappa }^{}  L_{\beta}$.

\justify
For $\alpha < \kappa$ limit let $L_{\alpha}: = \bigcup\limits_{\beta < \kappa }^{}  L_{\beta}$. By compactness there is $J_{\alpha}h_{\alpha}$ such that $d'_{<\alpha}\frown J_{\alpha}\equiv_{e}J$ and $h_{\alpha}J_{\alpha}\equiv_{e}h_{0}J$, then we let $d'_{\alpha}$ be the first term of $J_{\alpha}$.

\justify
Assume that $L_{\leq \alpha}d'_{\leq \alpha}J_{\leq \alpha}$ and $h_{\leq \alpha}$ have been constructed. By the Standard Lemma there is $J_{\alpha +1}$ such that $J_{\alpha +1}\equiv_{ed'_{\leq \alpha}}J_{\alpha>0}$ and $J_{\alpha +1}$ is $L_{\leq \alpha}h_{\leq \alpha}d'_{\leq \alpha}$-indiscernible. Then $d'_{<\alpha+1}\frown J_{\alpha+1}\equiv_{e}J$. So $L_{\alpha}h_{\leq \alpha}d'_{\alpha}\forkindep^{K}_{e}J_{\alpha+1}$, and by extension there is $h_{\alpha+1}$ such that $h_{\alpha+1}J_{\alpha+1}\equiv_{e}h_{0}J$ and $L_{\alpha}h_{\leq \alpha}d'_{\alpha}\forkindep^{K}_{e}h_{\alpha+1}J_{\alpha+1}$. Then there is an $e$-Morley sequence $L_{\alpha+1}=(L^{i}_{\alpha}h^{i}_{\alpha}d^{i}_{\alpha}$ : $i<\omega)$ such that $L^{0}_{\alpha}h^{0}_{\alpha}d^{0}_{\alpha}=L_{\alpha}h_{\alpha}d'_{\alpha}$ and that $L_{\alpha+1}$ is $eh_{\alpha+1}J_{\alpha+1}$-indiscernible.

\justify
For every $\alpha < \kappa$ $L_{\alpha}$ is a tree indexed on $\mathcal{T}_{\alpha}$ weakly spread out and s-indiscernible over $e$ such that : 

\begin{enumerate}
\item[•] If $\langle (uv)_{\eta_{\beta} }$ : $ \beta \leq \alpha \rangle $ with $dom(\eta_{\beta}) = [\beta, \alpha + 1)$ is an arbitrary path in the tree, then $\langle v_{\eta_{\beta} }$ : $ \beta \leq \alpha \rangle $ has the same EM-type as $J \equiv_{e} I$ over $e$, in particular any countable increasing subsequence is elementary equivalent to $I$ over $e$.
\item[•] For any $\beta \leq \alpha $ we have $ (uv)_{\eta_{\beta} }\equiv_{e}bc_{0}$.
\item[•] For any $\gamma$ with
$\beta < \gamma\leq \alpha $, we have $u_{\eta_{\beta}}v_{\eta_{\gamma}} \equiv_{e} bc_{1}$.
\end{enumerate} 

\justify
Consequently, for a large enough $\kappa$ when we shrink the tree into a weakly Morley tree using \cref{morleytree0} we can preserve the above conditions and the resulting weakly Morley tree also meets the conditions. Therefore we can find a weak tree Morley sequence as described in this lemma.\end{proof}

\justify
This theorem is an adaptation of \cite[Theorem 5.6]{dobrowolski2022independence} :

\begin{theorem}\label{5.6} Let $e \in \mathbb{M}^{heq}$ and $a,a',b,c$ be real tuples. Assume that $b\forkindep_{e}^{K} c$, $a \equiv_{e}^{L} a'$, and $a\forkindep_{e}^{K} b$ with $p(x, b) = \tp(a/eb)$, $a'\forkindep_{e}^{K} c$ with $q(x, c) = \tp(a'/ec)$. Then $p(x,b) \cup q(x, c)$ does not Kim-divide over $e$. 

\end{theorem}

\justify
\begin{proof} We can assume that $e \in dcl^{heq}(b)\cap dcl^{heq}(c)$ by normality and symmetry (by replacing $a,b$ and $c$ by $ea,eb$ and $ec$). Since $a \equiv^{L}_{e} a'$, there is an automorphism $f \in Autf(\mathbb{M}/e)$ sending $a'$ to $a$. Then clearly $a \models p(x, b) \cup q(x, c'')$ where $c'' = f(c)$, so $c \equiv^{L}_{e} c''$. In particular $a\forkindep_{e}^{K} c''$. Due to \cref{5.4}, we can assume that $b\forkindep_{e} c''$ and $a\forkindep_{e}^{K} bc''$, so $p(x, b) \cup q(x, c'')$ does not Kim-divide over $e$.

\justify
Now applying \cref{5.4} again to $b, c, c''$ we can
find $c'$ such that $bc'' \equiv_{e} bc'$, $c'\equiv^{L}_{e}c''  \equiv^{L}_{e} c$, $b\forkindep_{e}^{K}cc'$ and $c\forkindep_{e}^{K}c'$. Now
as pointed out in \cref{lascd}, there are $c'= c_{0}, c_{1}, . . . , c_{n} = c$ such that each pair
$c_{i}c_{i+1}$ starts a $e$-Morley sequence in $\tp(c_{i}/e)$. Moreover, due to extension of $\forkindep^{K}$ applied to $b\forkindep_{e}^{K}cc'$, by moving $b$ over $cc'$ we can assume that $b\forkindep_{e}^{K}c_{\leq n}$.

\justify
Recall that $p(x, b) \cup q(x, c_{0} )$ does not Kim-divide over $e$ (*), we shall show that $p(x, b) \cup q(x, c_{1} )$ does not Kim-divide over $e$. Then the same iterative argument will shows that each of $p(x, b) \cup q(x, c_{2} )$, . . . , $p(x, b) \cup q(x, c_{n} )$ does not Kim-divide over $e$ either, as wanted.

\justify
Now due to \cref{5.5}, there is a weak tree Morley sequence $I = (b'_{i}c'_{i}$ : $i<\omega)$ over $e$ such that $b'_{i}c'_{i} \equiv_{e}bc_{0}$ for any $i<\omega$ and $b'_{i}c'_{j} \equiv_{e}bc_{1}$ for any $j > i$. Then due to (*) and \cref{conjmorleytreebis}, $\lbrace p(x,b'_{i}) \cup q(x,c'_{i}) $ : $i<\omega \rbrace$  is consistent.

\justify
In particular, $\bigcup\limits_{i< \omega }^{}  p(x,b'_{2i}) \cup q(x,c'_{2i+1}) $ is consistent. Since $(b'_{2i}c'_{2i+1}$ : $i<\omega)$ is a weak tree Morley sequence over $e$ as well by \cref{treemorleysequ}, we have proved that $p(x, b)\cup q(x, c_{1})$ does not Kim-divide, and we can just repeat this argument starting with $p(x, b)\cup q(x, c_{1})$.\end{proof}

\begin{lemma}\label{lemmaamalg} Let $a,a',b,c,e \in \mathbb{M}^{heq}$ such that $b\forkindep^{K}_{e}c$, $a\equiv^{L}_{e}a'$, $a\forkindep^{K}_{e}b$ and $a'\forkindep^{K}_{e}c$. Then there are real tuples $\overline{a},\overline{a}',\overline{b},\overline{c}$ such that $\overline{b}\forkindep^{K}_{e}\overline{c}$, $\overline{a}\equiv^{L}_{e}\overline{a}'$, $\overline{a}\forkindep^{K}_{e}\overline{b}$, $\overline{a}'\forkindep^{K}_{e}\overline{c}$ and $bdd(ea)\in dcl^{heq}(\overline{a})$, $bdd(ea')\in dcl^{heq}(\overline{a})$, $bdd(eb)\in dcl^{heq}(\overline{b})$, $bdd(ec)\in dcl^{heq}(\overline{c})$.
\end{lemma}

\begin{proof}
By \cref{bondedkimmorley} we can replace $a,a',b,c$ and $e$ by $bdd(ea),bdd(ea'),bdd(eb),bdd(ec)$ and $bdd(e)$ respectively, which we can see as hyperimaginaries. By \cref{liftkimmorley} let $\overline{b},\overline{c}$ be representatives of $bdd(eb)$ and $bdd(ec)$ respectly such that $\overline{b}\forkindep^{K}_{e}\overline{c}$. Then by extension there is : \begin{center}
$\left\{ \begin{array}{l}
        \ a_{0}\equiv^{L}_{bdd(eb)} bdd(ea)$ such that $a_{0}\forkindep^{K}_{e}\overline{b}\\
        \ a'_{0}\equiv^{L}_{bdd(ec)} bdd(ea')$ such that $a'_{0}\forkindep^{K}_{e}\overline{c}
    \end{array}
    \right. $,
\end{center} (the equality of Lacar strong types just comes from the fact that these hyperimaginaries are boundedly closed) in particular we have that $a_{0}\equiv^{L}_{e}a'_{0}$, and we can replace $bdd(eb)$ and $bdd(ec)$ by $\overline{b}$ and $\overline{c}$.

\justify 
Now we choose a representative $\overline{a}_{0}$ of $a_{0}$ such that $\overline{a}_{0}\forkindep^{K}_{e}\overline{b}$. We can choose a representative $\overline{a}'_{1}$ of $a'_{0}$ such that $a_{0}\overline{a}_{0}\equiv^{L}_{e}a'_{0}\overline{a}'_{1}$. Then by extension there is $\overline{c}'$ such that $\overline{c}'\equiv^{L}_{bdd(ea')}\overline{c}$ and $\overline{a}'_{1}\forkindep^{K}_{e}\overline{c}'$, by moving $\overline{c}'$ to $\overline{c}$ over $bdd(ea')$ we now get $\overline{a}'_{0}\equiv^{L}_{ea'}\overline{a}'_{1}$ such that $\overline{a}'_{0}\forkindep^{K}_{e}c$ and we still have $\overline{a}'_{0}\equiv^{L}_{e}\overline{a}_{0}$.
\end{proof}

\begin{theorem}\label{amalghyp} (3-amalgamation of Lstp for Kim-dividing) Let $a,a',b,c,e \in \mathbb{M}^{heq}$. Let $b\forkindep^{K}_{e}c$, $a\equiv^{L}_{e}a'$, $a\forkindep^{K}_{e}b$ and $a'\forkindep^{K}_{e}c$. Then there is $a''$ such that $a''\forkindep^{K}_{e}bc$ and $a''\equiv^{L}_{eb}a$, $a''\equiv^{L}_{ec}a'$.
\end{theorem}

\begin{proof}
\justify
We begin by reducing to the case of real tuples using \cref{lemmaamalg} (we keep the notations) : 
    
\justify
We can apply \cref{5.6} to find $\overline{a}'' \models \tp(\overline{a}/\overline{b}) \cup \tp(\overline{a}'/\overline{c})$ such that $\overline{a}''\forkindep^{K}_{e}\overline{b}\overline{c}$. Now by the choice of $\overline{a},\overline{a}',\overline{b}$ and $\overline{c}$ we have $a''\equiv^{L}_{eb}a$ and $a''\equiv^{L}_{ec}a'$ with $a''=\overline{a}''_{E}$ and clearly $a''\forkindep^{K}_{e}bc$, as we wanted.\end{proof}

\justify
We now exploit our new amalgamation result. As for the case of simple theories (with \cite[Proposition 10.10]{casanovas2011simple} for exemple), we have the following :

\begin{cor}\label{remarkmorley} Let $e \in \mathbb{M}^{heq}$, $I = (a_{i}$ : $i<\omega)$ a Kim-Morley sequence of hyperimaginaries over $e$ and $p_{i}$ a partial type over $a_{i}$ that does not Kim-fork over $e$ such that for any realisations $b_{i} \models p_{i}$, $b_{j} \models p_{j}$ we have $b_{i}\equiv^{L}_{e}b_{j} $. Then $\lbrace  p_{i} (x)$ : $i<\omega \rbrace$ does not Kim-fork over $e$.
\end{cor}

\begin{proof}
\justify
We show by induction on $n <\omega$ that $\lbrace  p_{i} (x)$ : $i\leq n \rbrace$ does not Kim-divide over $e$. For $n=0$ it is clear, if it is true for $n \geq 0$ we have a realisation $b \models \lbrace  p_{i} (x)$ : $i\leq n \rbrace$ such that $b\forkindep^{K}_{e}a_{\leq n}$. By hypothesis there is $b' \models  p_{n+1} (x)$ such that $b'\forkindep^{K}_{e}a_{n+1}$. Since $I$ is Kim-Morley over $e$ $a_{n+1}\forkindep^{K}_{e}a_{\leq n}$ and by hypothesis $b\equiv^{L}_{e}b' $. By \cref{amalghyp} there is $b''  \models \lbrace  p_{i} (x)$ : $i\leq n+1 \rbrace$ such that $b'\forkindep^{K}_{e}a_{\leq n+1}$, and so $\lbrace  p_{i} (x)$ : $i\leq n+1 \rbrace$ does not Kim-divide over $e$, which finishes our induction, and also the proof by the local character of $\forkindep^{K}$.
\end{proof}

\justify
The following is \cite[Corollary 6.6]{kaplan2020kim}, an important result, and here we also follow the same track as in simple theories (see 10.11 of \cite{casanovas2011simple}) :

\begin{cor}\label{newsequence} Let $b,b',e \in \mathbb{M}^{heq}$, if $b\forkindep^{K}_{e}b'$ and $b\equiv^{L}_{e}b'$ then there is a weak tree Morley sequence $(b_{i}$ : $i<\omega)$ over $e$, with $b_{0} = b$ and $b_{1} = b'$.
\end{cor}

\begin{proof}
\justify
By \cref{lemmaamalg} we can assume that $b$ and $b'$ are real tuples, that $e = bdd(e)$ and $e\in dcl^{heq}(b)\cap dcl^{heq}(b')$, in particular $\equiv_{e}$ $ =$ $ \equiv^{L}_{e}$. Let $p(x, b) = \tp(b'/eb)$. Clearly, proving the result for these real tuples is enough since we can then reduce the weak tree Morley sequence to obtain the desired one for hyperimaginaries. The proof of the result for these real tuples carries over verbatim.
\end{proof}

\justify
The following is a strengthening of \cref{amalghyp} (it is \cite[Lemma 1.19]{chernikov2023transitivitylownessranksnsop1}) :

\begin{lemma}\label{hypamalg2} Let $a_{0},a_{1},b,c,e \in \mathbb{M}^{heq}$. Let $b\forkindep^{K}_{e}c$, $a_{0}\equiv^{L}_{e}a_{1}$, $a_{0}\forkindep^{K}_{e}b$ and $a_{1}\forkindep^{K}_{e}c$. Then there is $a$ such that $a\equiv^{L}_{eb}a_{0}$, $a\equiv^{L}_{ec}a_{1}$ and additionally $a\forkindep^{K}_{e}bc$, $b\forkindep^{K}_{e}ac$ and $c\forkindep^{K}_{e}ab$. 
\end{lemma}

\begin{proof}
\justify
By \cref{lemmaamalg} we can assume that $a_{0},a_{1},b$ and $c$ are real tuples. Once we have replaced hyperimaginaries with real tuples the proof carries over verbatim using \cref{newsequence}, \cref{amalghyp}, \cref{conjmorleytreebis}, \cref{bondedkimmorley} and \cref{remarkmorleytree}. We then take the reductions.\end{proof}

\subsection{Transitivity of Kim-Forking for hyperimaginaries}

\justify
This section consists in rewritings of the proofs of section 3 of \cite{chernikov2023transitivitylownessranksnsop1}. The proofs of \cref{p732h} and \cref{transitkim} are new. This article went throught an update as a gap concerning the construction of Morley trees was discovered, this is explained in details in \cite{chernikov2023transitivitylownessranksnsop1} in the paragraph following \cite[Definition 2.13]{chernikov2023transitivitylownessranksnsop1}.

\subsubsection{A small gap: Preliminary lemmas}

\justify
We begin by writing the adaptations of \cite[Lemmas 2.15 and 2.16]{chernikov2023transitivitylownessranksnsop1} to the hyperimaginary context. These lemmas were originally written to fix the gap that was present in the proofs of \cite[Lemma 3.1]{chernikov2023transitivitylownessranksnsop1} and \cite[Proposition 3.3]{chernikov2023transitivitylownessranksnsop1}, which correspond here to \cref{2.1trans} and \cref{p733h}.

\begin{definition} \cite[Definition 2.14]{chernikov2023transitivitylownessranksnsop1} Let $e$ be an hyperimaginary. A sequence $\langle (a_{\eta ,i})_{\eta\in \mathcal{T}_{\alpha}}\ :\ i<\kappa\rangle$ is mutually $s$-indiscenible over $e$ if, for all $i<\kappa$, $(a_{\eta ,i})_{\eta\in \mathcal{T}_{\alpha}}$ is $s$-indiscenible over $e\lbrace (a_{\eta ,j})_{\eta\in \mathcal{T}_{\alpha}}\ :\ j\neq i\rbrace$.
\end{definition}

\begin{lemma}\cite[Lemma 2.15]{chernikov2023transitivitylownessranksnsop1}\label{fix1} Let $e$ be an hyperimaginary. Given a tree $(a_{\eta})_{\eta\in \mathcal{T}_{\alpha}}$ that is $s$-indiscenible over $e$ there is a sequence $I=\langle (a_{\eta ,i})_{\eta\in \mathcal{T}_{\alpha}}\ :\ i<\omega\rangle$ such that $(a_{\eta ,0})_{\eta\in \mathcal{T}_{\alpha}} = (a_{\eta })_{\eta\in \mathcal{T}_{\alpha}}$, $I$ is an $e$-Morley sequence and is mutually $s$-indiscenible over $e$.\end{lemma}

\begin{proof} The proof follows the same steps as the one of \cite[Lemma 2.15]{chernikov2023transitivitylownessranksnsop1}. We present a short version here just to point out some additional steps that hyperimaginaries require us to make. First of all using \cref{modprop} and \cref{morleytree1} we can assume that $(a_{\eta})_{\eta\in \mathcal{T}_{\alpha}}$ is a tree of real tuples each of which contains a representative of $e$. The inductive construction of $I_{\gamma}=\langle (a_{\eta ,i})_{\eta\in \mathcal{T}_{\alpha}}\ :\ i<\gamma \rangle$ works similarly, the only thing to check is the local basedness argument:

\justify
Let $(b_{\eta })_{\eta\in \mathcal{T}_{\alpha}}\equiv_{e}(a_{\eta })_{\eta\in \mathcal{T}_{\alpha}}$ such that $(b_{\eta })_{\eta\in \mathcal{T}_{\alpha}} \forkindep^{f}_{e}I_{\gamma}$. We have that:

%Assume that a formula $\varphi(x;\overline{a})$ forks over $e$ for $\overline{a}\in I_{\gamma}$ and $x\in (x_{\eta })_{\eta\in \mathcal{T}_{\alpha}}$. Then none of the tuples of $(b_{\eta })_{\eta\in \mathcal{T}_{\alpha}}$ satisfy this formula as $(b_{\eta })_{\eta\in \mathcal{T}_{\alpha}} \forkindep^{f}_{e}I_{\gamma}$, which means that:

\begin{center}
$\lbrace \neg \varphi(x;\overline{a})\ :\ \varphi(x;\overline{a})$ forks over $e$ and $ \overline{a}\in I_{\gamma}\rbrace \subseteq EM((b_{\eta })_{\eta\in \mathcal{T}_{\alpha}}/eI_{\gamma})$
\end{center}

\justify
This means that if we take $(b'_{\eta })_{\eta\in \mathcal{T}_{\alpha}}$ that satisfies the $EM$-type of $(b_{\eta })_{\eta\in \mathcal{T}_{\alpha}}$ over $eI_{\gamma}$ and be $s$-indiscenible over $eI_{\gamma}$, then by local character $(b'_{\eta })_{\eta\in \mathcal{T}_{\alpha}} \forkindep^{f}_{e}I_{\gamma}$. Notice that for this replacing hyperimaginaries in the trees by real tuples was crucial. With this cleared out the rest of the argument and inductive construction of $(a'_{\eta,i })_{\eta\in \mathcal{T}_{\alpha}}$ carries on smoothly.\end{proof}

\begin{lemma}\cite[Lemma 2.16]{chernikov2023transitivitylownessranksnsop1}\label{fix2} Let $(a_{\eta})_{\eta\in \mathcal{T}_{\alpha+1}}$ be a tree of tuples such that $I:= \langle a_{\unrhd (i)}\ :\ i<\omega\rangle$ is mutually $s$-indiscernible and Morley over $e$. Let $(a'_{\eta})_{\eta\in \mathcal{T}_{\alpha+1}}$ be an $s$-indiscernible tree that satisfies the $EM$-type of $(a_{\eta})_{\eta\in \mathcal{T}_{\alpha+1}}$ over $e$. Then the sequence $I':= \langle a'_{\unrhd (i)}\ :\ i<\omega\rangle$ has the same type as $I$ over $e$ (i.e. the $EM$-type over $e$ restricted to $\mathcal{T}_{\alpha+1}\setminus \emptyset$ is complete), and thus it is also mutually $s$-indiscernible and Morley over $e$.
\end{lemma}

\begin{proof} The proof is striclty similar as the one of \cite[Lemma 2.16]{chernikov2023transitivitylownessranksnsop1} as it consists of manipulations of indexes.\end{proof}

\subsubsection{Properties of weak tree Morley sequences}

\justify

 The following is \cite[Lemma 3.1]{chernikov2023transitivitylownessranksnsop1}, the proof carries over verbatim using \cref{fix1} and \cref{fix2}.

\begin{lemma}\label{2.1trans} Let $a_{0},b,e \in \mathbb{M}^{heq}$. If $a_{0}\forkindep^{K}_{e}b$, $e\in dcl^{heq}(b)$ then there is a weak tree Morley sequence $( a_{i}$ : $i<\omega )$ over $b$ such that $a_{i}\forkindep^{K}_{e}ba_{<i}$ for all $i<\omega$.\end{lemma}

\begin{lemma}\label{p732h} Let $a,b,c,e \in \mathbb{M}^{heq}$. If $a\forkindep^{K}_{e}b$ and $c\forkindep^{K}_{e}b$ then there is $c'$ such that $c'\equiv_{eb}c$, $ac'\forkindep^{K}_{e}b$ and $a\forkindep^{K}_{eb}c'$.\end{lemma}

\begin{proof}

\justify
By \cref{lemmaamalg} there are real tuples $\overline{b}, \overline{a}, \overline{c}$ such that $\overline{a}\forkindep^{K}_{e}\overline{b}$, $\overline{c}\forkindep^{K}_{e}\overline{b}$, and $b,e \in dcl^{heq}(\overline{b})$, $a,e \in dcl^{heq}(\overline{a})$, $c,e \in dcl^{heq}(\overline{c})$. Let $\Gamma'(x,\overline{b},y)$ be the type that defines the set of $\overline{a}'\overline{c}'$ such that $\overline{a}'\overline{c}'\forkindep^{K}_{e}b$ for any $\overline{b}$ representative of $b$ (\cref{typedef1}). Let $\Gamma (x,\overline{b},\overline{a}) =\Gamma' (x,\overline{b},\overline{a}) \cup \tp(\overline{c}/eb)$.

\justify
By \cref{2.1trans} there is $(\overline{a}_{i}$ : $i<\omega ) $ is an $eb$-indiscernible sequence satisfying $\overline{a}_{0}=\overline{a}$ and $\overline{a}_{i}\forkindep^{K}_{e}b\overline{a}_{<i}$ for all $i<\omega $. By \cref{morleytree1} we can complete it to a weak tree Morley sequence $(\overline{a}_{i}\overline{b}_{i}$ : $i<\omega )$ with $\overline{b}_{i}$ a representative of $b$. We show by induction on $n < \omega $ that we can find $\overline{c}_{n}\equiv^{L}_{e}\overline{c}$ such that $\overline{c}_{n}\forkindep^{K}_{e}b\overline{a}_{<n}$ and $\overline{c}_{n}\overline{a}_{m}\forkindep^{K}_{e}b\overline{a}_{<m}$ for all $m<n$ : For $n=0$ we take $\overline{c}_{0} =\overline{c}$.

\justify
Assume that we have found $\overline{c}_{n}$. So we have $\overline{c}_{n}\equiv^{L}_{e}\overline{c}$ and $\overline{c}_{n}\forkindep^{K}_{e}b\overline{a}_{<n}$. By \cref{exthyp} there is $\overline{c}'\equiv_{e}^{L}\overline{c}$ such that $\overline{c}'\forkindep^{K}_{e}\overline{a}_{n}$.
Then $\overline{c}'\equiv_{e}^{L}\overline{c}\equiv_{e}^{L}\overline{c}_{n}$ and since $\overline{a}_{n}\forkindep^{K}_{e}b\overline{a}_{<n}$ by \cref{hypamalg2} we can find $\overline{c}_{n+1}\equiv_{e}^{L}\overline{c}$ such that $\overline{c}_{n+1} \models \tp(\overline{c}_{n}/eb\overline{a}_{<n})\cup \tp(\overline{c}'/e\overline{a}_{n})$, $\overline{c}_{n+1}\forkindep^{K}_{e}b\overline{a}_{<n+1}$ and $\overline{c}_{n+1}\overline{a}_{n}\forkindep^{K}_{e}b\overline{a}_{<n}$, which completes our induction. Then $\overline{c}_{n+1} \models \bigcup\limits_{i< n+1 }^{} \Gamma (x,\overline{b}_{i},\overline{a}_{i})$, so $\bigcup\limits_{i< \omega }^{} \Gamma (x,\overline{b}_{i},\overline{a}_{i}) $ is consistent.

\justify
By \cref{kimlemmatree} $\Gamma (x,\overline{b}_{0},\overline{a}_{0})$ does not Kim-fork over $b$, so there is $\overline{c}' \models \Gamma (x,\overline{b}_{0},\overline{a}_{0})$ such that $\overline{c}'\forkindep^{K}_{b}\overline{a}_{0}$, which gives the desired hyperimaginary after reducing by the right equivalence relation.\end{proof}

\justify
The following is \cite[Proposition 3.3]{chernikov2023transitivitylownessranksnsop1}, the proof carries over verbatim using \cref{p732h}, \cref{fix1} and \cref{fix2}.

\begin{prop}\label{p733h} Let $a_{0},b,e \in \mathbb{M}^{heq}$. If $a_{0}\forkindep^{K}_{e}b$, then  there is an $eb$-indiscernible sequence $I=( a_{i}$ : $i<\omega )$ which is a weak tree Morley sequence over $e$ and a Kim-Morley sequence over $eb$.\end{prop}

\begin{theorem}\label{transitkim} Transitivity of Kim-independence : Let $a,b,c,e \in \mathbb{M}^{heq}$. If $e \in dcl^{heq}(b)$, $a\forkindep^{K}_{e}b$ and $a\forkindep^{K}_{b}c$, then $a\forkindep^{K}_{e}bc$.\end{theorem}

\justify
\begin{proof}

\justify
Since $a\forkindep^{K}_{e}b$ by \cref{p733h} there is a sequence $I = ( a_{i}$ : $i< \omega )$ with $a_{0} = a$ such that $I$ is a Kim-Morley sequence over $b$ and a weak tree Morley sequence over $e$. By symmetry $c\forkindep^{K}_{b}a$, so by the chain condition \cref{kimdivh} there is $I' \equiv_{ba} I $ such that $I' = ( a'_{i}$ : $ i< \omega )$ is $bc$-indiscernible. $I'$ is then also a weak tree Morley sequence over $e$ such that $a'_{0}=a$. Then by \cref{remarkmorleytree2bis} $bc\forkindep^{K}_{e}a$, and we conclude by symmetry.\end{proof}

\subsection{Witnessing over hyperimaginaries}

\justify
Thanks to transitivity we can now prove the following, which is \cite[Lemma 2.3]{kim2021weak} and is similar to \cref{kfhyp} but for Kim-Morley sequences, the proof is also similar : 

\begin{prop}\label{reprkimmorley} Let $e\in \mathbb{M}^{heq}$ and $I = ( a_{i}$ : $i<\omega )$ be a Kim-Morley sequence over $e$. For every $b\in \mathbb{M}^{heq}$ there is a Kim-Morley sequence $( a_{i}b_{i}$ : $i<\omega )$ over $e$ such that $b\equiv_{ea_{0}}b_{0}$.\end{prop}

\justify
The following results are from section 3 of \cite{chernikov2023transitivitylownessranksnsop1}. This is an adaptation of \cite[Proposition 3.5]{chernikov2023transitivitylownessranksnsop1} :

\begin{prop} Let $e\in \mathbb{M}^{heq}$, the following are equivalent : 
\begin{enumerate}
\item[(1)]  $a\forkindep^{K}_{e}b$
\item[(2)]  There is a model $M$ such that $e \in dcl^{heq}(M)$, $M\forkindep^{K}_{e}ab$ (or $M\forkindep_{e}ab$) and $a\forkindep^{K}_{M}b$.
\item[(3)] There is a model $M$ such that $e \in dcl^{heq}(M)$, $M\forkindep^{K}_{e}a$ (or $M\forkindep_{e}a$) and $a\forkindep^{K}_{M}b$.
\end{enumerate}

\end{prop}

\justify
\begin{proof}
$(1)\Rightarrow(2)$ : By \cref{liftkimmorley} we can replace $a$ and $b$ by representatives. Since  $a\forkindep^{K}_{e}b$, there is a Morley sequence $I = (a_{i}$ : $i<\omega)$ over $e$ with $a_{0} = a$ such that $I$ is $eb$-indiscernible. By \cref{hyp6} (which we proved for real tuples), there is a
model $N$ such that $e \in dcl^{heq}(N)$, $N\forkindep_{e}I$ and $I$ is a coheir sequence over $N$. By compactness and extension we can clearly assume the length of $I$ is arbitrarily large, and $N\forkindep_{e}Ib$. Hence by the pigeonhole principle, there is an infinite subsequence
$J$ of $I$ such that all the elements in $J$ have the same type over $Nb$. Thus, for $a' \in J$, we have $a'\forkindep^{K}_{N}b$ (by Kim's Lemma for Kim-forking) and $N\forkindep_{e}a'b$. Hence $M = f(N)$ is a desired model, where $f$ is
an $eb$-automorphism sending $a'$ to $a$. $(2)\Rightarrow(3)$ is clear, and $(3)\Rightarrow(1)$ follows from transitivity.\end{proof}

\justify
In the following proposition by an increasing continuous sequence  of hyperimaginaries $(e_{i}$ : $i<\kappa^{+})$ we mean that if $\alpha < \beta < \kappa^{+}$ then $e_{\alpha} \in dcl^{heq}(e_{\beta})$ and if $\lambda < \kappa^{+}$ is a limit ordinal, then $e_{\lambda} \sim (e_{i}$ : $i<\lambda)$. By 'union $e$' in (3) we mean that $e_{\lambda} \sim (e_{i}$ : $i<\kappa^{+})$. This is an adaptation of \cite[Proposition 3.6]{chernikov2023transitivitylownessranksnsop1} :

\begin{prop} Let $\kappa \geq \vert T \vert$ be a cardinal. The following are equivalent.
\begin{enumerate}
\item[(1)] $T$ is $NSOP_{1}$
\item[(2)]  There is no increasing continuous sequence  of hyperimaginaries $(e_{i}$ : $i<\kappa^{+})$ and finite hyperimaginary $d$ such that $\vert e_{i}\vert \leq \kappa$ and $d\centernot{\forkindep}^{K}_{e_{i}}e_{i+1}$ for all $i<\kappa^{+}$.
\item[(3)] There is no $e\in \mathbb{M}^{heq}$ such that $\vert e \vert = \kappa^{+}$ and $p(x) \in S_{E}(e)$ with $x$ a finite tuple of variables such that for some increasing and continuous sequence $(e_{i}$ : $i<\kappa^{+})$ of hyperimaginaries with union $e$, we have $\vert e_{i}\vert \leq \kappa$ and $p$ Kim-divides over
$e_{i}$ for all $i<\kappa^{+}$.
\end{enumerate}

\end{prop}

\justify
\begin{proof}
$(1)\Rightarrow(2)$ : By replacing $d$ by a representative we can assume that $d$ is a real tuple. Assume that there is an increasing continuous sequence
of hyperimaginaries $(e_{i}$ : $i<\kappa^{+})$ and a finite tuple $d$ such that $\vert e_{i}\vert \leq \kappa$ and $d\centernot{\forkindep}^{K}_{e_{i}}e_{i+1}$ for all $i<\kappa^{+}$. We show that there is a continuous increasing sequence of models $(M_{i}$ : $i<\kappa^{+})$ and a tuple $d'$ such that $\vert M_{i}\vert \leq \kappa$ and $d'\centernot{\forkindep}^{K}_{M_{i}}M_{i+1}$ for all $i<\kappa^{+}$. This is enough as such a sequence of models implies that T has SOP1 by \cite[Theorem 3.9]{kaplan2019local}.

\justify
By induction on $i < \kappa^{+}$ we will build two increasing and continuous
sequences $(e'_{i}$ : $i<\kappa^{+})$ and $(M_{i}$ : $i<\kappa^{+})$ satisfying the following for all $i<\kappa^{+}$ :
\begin{enumerate}
\item[(a)] $e'_{0}=e_{0}$ and $e'_{\leq i}\equiv e_{\leq i}$.
\item[(b)] $M_{i} \models T$ with $\vert M_{i} \vert = \kappa$ and $e'_{i} \in dcl^{heq}(M_{i})$.
\item[(c)] $e'_{i+1}\forkindep^{K}_{e'_{i}}M_{i}$.
\end{enumerate}

\justify
Let $e'_{0}:=e_{0}$ and $M_{0}$ be any model of size $\kappa$ such that $e'_{0} \in dcl^{heq}(M_{0})$. Given $e'_{\leq i}$ and $M_{\leq i}$ satisfying the requirements, we pick $e''_{i+1}$ such that $e'_{\leq i}e''_{i+1}\equiv e_{\leq i+1}$. Then we apply extension to obtain $e'_{i+1}\equiv_{e'_{\leq i}} e''_{i+1}$ such that $e'_{i+1}\forkindep^{K}_{e'_{i}}M_{i}$ (this is due to the fact that $e_{i+1}\forkindep^{K}_{e_{i}}e_{\leq i}$, which is a consequence of existence and $e_{\leq i} \subseteq dcl^{heq}(e_{i})$). Note that $e'_{\leq i+1}\equiv e_{\leq i+1}$. We define $M_{i+1}$ to be any model
containing $e'_{i+1}M_{i}$ of size $\kappa$. This satisfies the requirements.

\justify
At limit $\delta$, we take $e'_{\delta}$ such that $e'_{\leq \delta}\equiv e_{\leq \delta}$ and $M_{\delta} = \bigcup\limits_{i< \delta }^{} M_{i}$. This satisfies (a), (b) is clear, and (c) empty. Therefore this completes the construction.

\justify
Let $M =\bigcup\limits_{i< \kappa^{+} }^{} M_{i}$. Choose $d'$ such that $d'e'_{< \kappa^{+} }\equiv de_{< \kappa^{+} }$, which is possible by (a). Then we have $d'\centernot{\forkindep}^{K}_{e'_{i}}e'_{i+1}$ for all $i<\kappa^{+}$. If there is some $i<\kappa^{+}$ such that $d'\forkindep^{K}_{M_{i}}M_{i+1}$ then we have $d'\forkindep^{K}_{M_{i}}e'_{i+1}$ by (b). Additionally, because $M_{i}\forkindep^{K}_{e'_{i}}e'_{i+1}$, we know, by symmetry and transitivity, that $e'_{i+1}\forkindep^{K}_{e'_{i}}d'M_{i}$. By symmetry once more, we get $d'\forkindep^{K}_{e'_{i}}e'_{i+1}$, a contradiction. This shows that $d'\centernot{\forkindep}^{K}_{M_{i}}M_{i+1}$ for all $i<\kappa^{+}$, completing the proof of this direction.

\justify
$(2)\Rightarrow(3)$ :  Suppose (3) fails, i.e. we are given $e$ of size $\kappa^{+}$, $p \in S_{E}(e)$, and an increasing and continuous sequence $(e_{i}$ : $i<\kappa^{+})$ of hyperimaginaries with union $e$ such that $\vert e_{i}\vert \leq \kappa$ and $p$ Kim-divides over $e_{i}$ for all $i<\kappa^{+}$. Let $d \models p$, we will define an increasing continuous sequence of ordinals  $(\alpha_{i}$ : $i<\kappa^{+})$ such that $\alpha_{i} \in \kappa^{+}$ and $d\centernot{\forkindep}^{K}_{e_{i}}e_{i+1}$ for all $i < \kappa^{+}$.

\justify
We set $\alpha_{0} = 0$. Given $(\alpha_{j}$ : $j \leq i)$ we know that there is some formula $\varphi(x, \overline{e}) \in p$ that Kim-divides over $e_{\alpha_{i}}$ by our assumption on $p$. Let $\alpha_{i+1}$ be the
least ordinal $<\kappa^{+}$ and $\geq \alpha_{i}$ such that $\tp(d/e_{\alpha +1}) \vdash \varphi(x, \overline{e}) $ (such an ordinal exist overwise we would have $d'\equiv_{e_{i}}d$ for all $i<\kappa^{+}$ such that $d'\centernot{\equiv}_{e}d$, contradicting the fact that $e$ is the union of the $e_{i}$). For limit $i$, if we are given $(\alpha_{j}$ : $j < i)$, we put $\alpha_{i} = sup_{j<i} \alpha_{j}$. Then we define $(e'_{i}$ : $i<\kappa^{+})$ by $e'_{i}= e_{\alpha_{i}}$ for all $i < \kappa^{+}$. By construction we have $d\centernot{\forkindep}^{K}_{e'_{i}}e'_{i+1}$ for all $i < \kappa^{+}$, which witnesses the failure of (2).

\justify
$(3)\Rightarrow(1)$ : This was established for real tuples as basis : \cite[Theorem 3.9]{kaplan2019local}.
\end{proof}

\justify
Kim-Morley sequences are witnesses to Kim-dividing over hyperimaginaries :

\begin{theorem}\label{witnessinghyp}
Witnessing : Let $e\in \mathbb{M}^{heq}$ and $I = (a_{i}$ : $i<\omega)$ be a Kim-Morley sequence of real tuples over $e$. If $\varphi (x,a_{0})$ Kim-divides over $e$ then $\lbrace \varphi (x,a_{i})$ : $i<\omega \rbrace$ is inconsistent.
\end{theorem}

\justify
\begin{proof}
Suppose towards contradiction that $\varphi(x, a_{0})$ Kim-divides over $e$ and $I = (a_{i}$ : $i<\omega)$ is an $e$-Morley sequence such that $\lbrace \varphi (x,a_{i})$ : $i<\omega \rbrace$ is consistent. We may stretch
$I$ such that $I = (a_{i}$ : $i<\kappa^{+})$ with $\vert T \vert +\vert e \vert \leq \kappa$. Let $e_{i}\sim ea_{<i}$. Then $(e_{i}$ : $i<\kappa^{+})$ is increasing and continuous and $\vert e_{i} \vert \leq \kappa$.

\justify
Let $d \models \lbrace \varphi (x,a_{i})$ : $i<\kappa^{+} \rbrace$. We claim that $d\centernot{\forkindep}^{K}_{e_{i}}e_{i+1}$ for all $i < \kappa^{+}$. If not, then for some  $i < \kappa^{+}$, we have $d\forkindep^{K}_{e_{i}}e_{i+1}$, or, in other words $d\forkindep^{K}_{ea_{<i}}a_{i}$. Since $I$ is an Kim-Morley sequence over $e$ we also have $a_{i}\forkindep^{K}_{e}a_{<i}$, hence $da_{<i}\forkindep^{K}_{e}a_{i}$ by transitivity. This entails, in particular, that $d\forkindep^{K}_{e}a_{i}$, which is a contradiction, since $\varphi(x, a_{i})$ Kim-divides over $e$. This completes the proof.
\end{proof}

\justify
From \cref{reprkimmorley} and \cref{witnessinghyp} we deduce :

\begin{cor}\label{witnessinghyp2}
Let $a_{0},e \in \mathbb{M}^{heq}$, $I = (a_{i}$ : $i<\omega)$ a Kim-Morley sequence of hyperimaginaries of sort $E$ over $e$ and $p(x,a_{0})$ a partial type over $a_{0}$ such that $\lbrace p(x,a_{i})$ : $i<\omega \rbrace$ is consistent. Then $p(x,a_{0})$ does not Kim-fork over $e$.
\end{cor}

\begin{cor}\label{hypwitn+} Let $e\in \mathbb{M}^{heq}$, $I = (a_{i}$ : $i<\omega)$ be a Kim-Morley sequence of hyperimaginaries over $e$ and $b\in \mathbb{M}^{heq}$ such that all the $a_{i}$ have the same type over $eb$. Then $b\forkindep^{K}_{e}a_{0}$.
\end{cor}

\justify
The converse is already known to be true as Morley sequences are Kim-Morley sequences. We had already proven the following for weak tree Morley sequences (\cref{remarkmorleytree2bis}) :

\begin{cor}\label{chaintotal0} Let $e\in \mathbb{M}^{heq}$, $I = (a_{i}$ : $i<\omega)$ a total Kim-Morley sequence of hyperimaginaries over $e$ and $b\in \mathbb{M}^{heq}$ such that $I$ is $eb$-indiscernible. Then $b\forkindep^{K}_{e}I$.
\end{cor}

\justify
\begin{proof}
By finite character to show that $b\forkindep^{K}_{e}I$ it is enough to show that $b\forkindep^{K}_{e}a_{<n}$ for all $n<\omega$. We have that $(a_{n\cdot i},...a_{n\cdot i + n-1}$ : $i<\omega)$ is Kim-Morley over $e$ by hypothesis and also $e\hat{b}$-indiscernible. Then, by \cref{hypwitn+} we have $b\forkindep^{K}_{e}a_{<n}$.
\end{proof}

\begin{cor}\label{chaintotal1} Chain condition for Kim-Morley sequences : Let $e\in \mathbb{M}^{heq}$, $I = (a_{i}$ : $i<\omega)$ be a Kim-Morley sequence of hyperimaginaries over $e$ and $b\in \mathbb{M}^{heq}$ such that $b\forkindep^{K}_{e}a_{0}$, then there is $b' \equiv^{L}_{ea_{0}}b$ such that $I$ is $eb'$-indiscernible and $b'\forkindep^{K}_{e}I$.
\end{cor}

\justify
\begin{proof}
We begin by extending $I$ to $(a_{i}$ : $i<\kappa)$ for a large enough $\kappa$. By naturality, symmetry and \cref{bondedkimmorley} we have $b\forkindep^{K}_{e}bdd(ea_{0})$. We can assume that $\alpha'_{0}= bdd(ea_{0})$ is a single hyperimaginary. Let $\overline{\alpha}'_{0}$ be a representative of $\alpha'_{0}$, by \cref{reprkimmorley} we have $I' = (a_{i}\overline{\alpha}_{i}$ : $i<\kappa)$ an Kim-Morley sequence over $e$ with $\overline{\alpha}_{0}\equiv_{ea_{0}}\overline{\alpha}'_{0} $, since an $ea_{0}$-automorphism globally fixes $bdd(ea_{0})$ we have $\alpha_{0}\sim\alpha'_{0} $ so we still have $b\forkindep^{K}_{e}\alpha_{0}$.

\justify
Now let $p(x,\overline{\alpha}_{0})= \tp(b/\alpha_{0})$. By \cref{remarkmorley} we have that $\lbrace p(x,\overline{\alpha}_{i})$ : $i< \kappa $ does not Kim-fork over $e$. So by extension there is $b''\models \lbrace p(x,\overline{\alpha}_{i})$ : $i< \kappa $ such that $b''\forkindep^{K}_{e}(\overline{\alpha}_{i}$ : $i<\kappa)$. By Erdös-Rado there is an $eb''$-indiscernible sequence $(\hat{\alpha}_{i}$ : $i<\omega)$ which is finitely based on $(\overline{\alpha}_{i}$ : $i<\kappa)$ over $eb''$. $b''\forkindep^{K}_{e}(\overline{\alpha}_{i}$ : $i<\kappa)$ so $b''\forkindep^{K}_{e}(\hat{\alpha}_{i}$ : $i<\omega)$. We have $b'$ such that $b'(\overline{\alpha}_{i}$ : $i<\omega) \equiv_{e} b''(\hat{\alpha}_{i}$ : $i<\omega)$. Then $\tp(b'\overline{\alpha}_{0})=\tp(b''\hat{\alpha}_{0})=\tp(b\overline{\alpha}_{0})$, so $b\equiv^{L}_{ea_{0}}b$.
\end{proof}

\justify
We now show that total Kim-Morley sequences are weak tree Morley sequences, we had already shown the converse in \cref{totalmorleyh}, using witnessing we can complete the equivalence.

\begin{definition} Let $e \in \mathbb{M}^{heq}$ and $(a_{i}$ : $i<\omega)$ be an $e$-indiscernible sequence of hyperimaginaries of sort $E$.
\begin{enumerate}
\item[(1)] Say $(a_{i}$ : $i<\omega)$ is a witness for Kim-dividing over $e$ if, whenever $\overline{a}_{0}$ is a representative of $a_{0}$ and $p(x, \overline{a}_{0})$ a partial type over $a_{0}$ Kim-divides over $e$, $\lbrace p(x, \overline{a}_{i})$ : $i < \omega\rbrace$ is inconsistent for some (any) sequence of representatives.
\item[(2)] Say $(a_{i}$ : $i<\omega)$ is a strong witness for Kim-dividing over $e$ if for all $n<\omega$, the sequence $(a_{[n\cdot i, n\cdot i+1)}$ : $i<\omega )$ is a witness to Kim-dividing over $e$.
\end{enumerate}
\end{definition}

\begin{lemma}\label{remarkstrongwitness}(\cref{witnessinghyp2}) : Kim Morley sequences over $e$ are witnesses for Kim-dividing over $e$ and so strong Kim-Morley sequences (so in particular weak tree Morley sequences) over $e$ are strong witnesses for Kim-dividing over $e$.
\end{lemma}

\begin{lemma}\label{kimlemmawitness} Let $I =(a_{i}$ : $i<\omega)$ be a witness of Kim-dividing over $e$. Then if $I$ is $eb$-indiscernible $b\forkindep^{K}_{e}a_{0}$, and if $I$ is a strong witness we have $b\forkindep^{K}_{e}I$.\end{lemma}

\justify
\begin{proof} Clear from the definition, and the second point by local character and the same argument as in \cref{chaintotal0}.\end{proof}

\justify
The following proposition is an adaptation of \cite[Proposition 7.9]{kaplan2020kim}.

\begin{prop}\label{thiskek} Let $e \in \mathbb{M}^{heq}$, then $(a_{i}$ : $i<\omega)$ is a strong witness for Kim-dividing over $e$ if and only if it is a weak tree Morley sequence over $e$.
\end{prop}

\justify
\begin{proof}
The right to left direction is given by \cref{remarkstrongwitness}. For the other direction the proof carries over verbatim using \cref{kimlemmawitness}, \cref{kimdivh} and \cref{morleytree0}.\end{proof}

\begin{cor}\label{totaltreemorley} Let $e\in \mathbb{M}^{heq}$. A sequence $I$ of hyperimaginaries is total Kim-Morley over $e$ if and only if it is a weak tree Morley sequence over $e$.
\end{cor}

\section{Weak Canonical Bases in NSOP1 Theories}

\justify
We will now adapt the results of \cite{kim2021weak} to the context of hyperimaginaries. In this article Kim uses the properties of Kim-Forking that we proved in the previous section for hyperimaginaries and does not return to the level of formulas. This means that similar proofs will work.

\subsection{More Properties of Kim-Forking}

\justify
First a reminder of some properties we use a lot :

\begin{fact}\label{doublemorley} Let $a_{0},b,e \in \mathbb{M}^{heq}$ with $e \in dcl^{heq}(b)$ such that $a_{0}\forkindep^{K}_{e}b$. Then there is a total Kim-Morley sequence over $e$ $J = (a_{i}$ : $i<\omega)$ such that $J\forkindep^{K}_{e}b$ and $J$ is Kim-Morley over $b$.
\end{fact}

\justify
This is just \cref{p733h} and \cref{totalmorleyh}. We will rather use the term total Kim-Morley sequences than weak tree Morley sequences since it corresponds to the use that will be made of it.

\begin{fact}\label{chaincondhyp} Chain condition for Kim-Morley sequences (\cref{chaintotal1}): Let $e\in \mathbb{M}^{heq}$, $I = (a_{i}$ : $i<\omega)$ a Kim-Morley sequence over $e$ of hyperimaginaries and $b\in \mathbb{M}^{heq}$ such that $b\forkindep^{K}_{e}a_{0}$. Then there is $b' \equiv^{L}_{ea_{0}}b$ such that $I$ is $eb'$-indiscernible and $b'\forkindep^{K}_{e}I$.
\end{fact}

\begin{fact}\label{chaintotal3} (\cref{chaintotal0}) Let $e\in \mathbb{M}^{heq}$, $I = (a_{i}$ : $i<\omega)$ a total Kim-Morley sequence of hyperimaginaries over $e$ and $b\in \mathbb{M}^{heq}$ such that $I$ is $eb$-indiscernible. Then $b\forkindep^{K}_{e}I$.
\end{fact}

\justify
We now give results about sequences that are Kim-Morley for two bases. The following is \cite[Proposition 2.2]{kim2021weak}, the proof carries over verbatim.

\begin{prop}\label{2.2wcb} Let $a_{0}\forkindep^{K}_{e}b$ with $e\in dcl^{heq}(b)$ and $I = ( a_{i}$ : $i<\omega )$ a Kim-Morley sequence over $e$. Then there is $I'\equiv_{ea_{0}}I$ such that $I'\forkindep^{K}_{e}b$ and $I'$ is Kim-Morley over $b$ (and Kim-Morley over $e$).\end{prop}

\justify
The following is \cite[Proposition 2.4]{kim2021weak}.

\begin{cor}\label{2.4wcb} Let $a_{0}\forkindep^{K}_{e}b$ with $e\in dcl^{heq}(b)$ and $I = ( a_{i}$ : $i<\omega )$ a Kim-Morley sequence over $e$. Then there is $I'\equiv^{L}_{ea_{0}}I$ such that $I'\forkindep^{K}_{e}b$ and $I'$ is Kim-Morley over $b$ (and Kim-Morley over $e$).\end{cor}

\begin{proof}
\justify
As in the proof of \cref{chaintotal1} we can replace $I = ( a_{i}$ : $i<\omega )$ by an $e$-Morley sequence $\tilde{I} = ( \alpha_{i}$ : $i<\omega )$ with $\alpha_{i} \sim bdd(ea_{i})$ for all $i<\omega$. We still have $\alpha_{0}\forkindep^{K}_{e}b$. Applying \cref{2.2wcb} to $\tilde{I}$ and $b$ gives us $\tilde{I}'\equiv_{\alpha_{0}}\tilde{I}$ such that $\tilde{I}'\forkindep^{K}_{e}b $ and $\tilde{I}'$ is Kim-Morley over $b$. Then there is $I'$ such that $I'\tilde{I}'\equiv_{\alpha_{0}}I\tilde{I}$, so in particular $I'\equiv^{L}_{ea_{0}}I$ and $I'$ clearly satisfies our conditions.
\end{proof}

\justify
The following is \cite[Theorem 2.5]{kim2021weak}, the proof carries over verbatim.

\begin{theorem}\label{2a2indeph}
Let $a,b_{0},c \in \mathbb{M}^{heq}$ be pairwise Kim-independent over $e\in \mathbb{M}^{heq}$. Then, there is $b$ such that  : 
\vspace{-10pt}
\begin{center}$b\equiv^{L}_{ea} b_{0}$, $b\equiv^{L}_{ec} b_{0}$, $b\forkindep^{K}_{ea}c$ and $b\forkindep^{K}_{ec}a$.\end{center}
\justify
So by transitivity of $\forkindep^{K}$ :  $a\forkindep^{K}_{e}bc$, $b\forkindep^{K}_{e}ac$ and $c\forkindep^{K}_{e}ab$.
\end{theorem}
    
\begin{theorem}\label{amalg1h} \cite[Theorem 2.7]{kim2021weak} : Let $a_{0}\forkindep^{K}_{b}a_{1}$, $c_{0}\equiv^{L}_{b} c_{1}$ and $c_{i}\forkindep^{K}_{b}a_{i} $ for $i=0,1$. Then there is $c\equiv^{L}_{ea_{i}}c_{i}$ for $i=0,1$ such that : \begin{center}
$c\forkindep^{K}_{b}a_{0}a_{1}$, $c\forkindep^{K}_{ba_{0}}a_{1}$, $c_{i}\forkindep^{K}_{ba_{1}}a_{0}$.
 \end{center}\end{theorem}

\begin{proof}
\justify
By type-amalgamation there is $c'\equiv^{L}_{ba_{i}}c_{i}$ for $i=0,1$ such that $c'\forkindep^{K}_{b}a_{0}a_{1}$. By \cref{2a2indeph} applied to $a_{0},c,a_{1}$ over $b$, we have $c\equiv^{L}_{ba_{i}}c_{i}$ for $i=0,1$ satisfying the conditions.\end{proof}

\justify
Two analogous results of \cref{2.2wcb} and \cref{2.4wcb} for total Kim-Morley sequences, \cref{2.7wcbnew} is new. The following is an adaptation of \cite[Theorem 2.14]{kim2021weak}, the proof carries over verbatim.

\begin{theorem}\label{2.7wcb}
Let $a_{0}\forkindep^{K}_{e}b$ with $e\in dcl^{heq}(b) $ and $I = ( a_{i}$ : $i<\omega)$ be total Kim-Morley over $e$. Then there is $I'\equiv_{ea_{0}}I$ such that $I'\forkindep^{K}_{e}b $ and $I'$ is total Kim-Morley over $b$ (and also over $e$).\end{theorem}

\begin{cor}\label{2.7wcbnew}
Let $a_{0}\forkindep^{K}_{e}b$ with $e\in dcl^{heq}(b) $ and $I = ( a_{i}$ : $i<\omega)$ be total Kim-Morley over $e$. Then there is $I'\equiv^{L}_{ea_{0}}I$ such that $I'\forkindep^{K}_{e}b $ and $I'$ is total Kim-Morley over $b$ (and also over $e$).\end{cor}

\begin{proof}
By the exact same reasoning as in \cref{2.4wcb} : we can change the sequence $I$ and apply \cref{2.7wcb} over boudedly closed hyperimaginaries.
\end{proof}

\subsection{Colinearity}

\begin{definition} For $I$ and $J$ indiscernible sequences $l(I,J)$ means that $I\frown J$ is indiscernible. For two given order type $\overline{x},\overline{y}$, $l(\overline{x},\overline{y})$ is type-definable.\end{definition}

\begin{definition} : \begin{enumerate}
\item We define $Q(I,I',J,J',L):=l(I,I')\wedge l(J,J') \wedge l(I',L) \wedge l(J',L)$.
\item $I\approx J$ means $Q(I,I',J,J',L)$ for some indiscernible sequences $I',J',L$.
\item  $Q_{0}(I,I',J,J',L)$ means that $Q(I,I',J,J',L)$ is true and that for any $K\in A=\lbrace I',J,J',L \rbrace $, $K$ is $IA\setminus K$-indiscernible.
\item  In the same fashion, $Q_{1}(I,I',J,J',L)$ means that $Q(I,I',J,J',L)$ is true and that for any $K\in B=\lbrace I,I',J',L \rbrace $, $K$ is $JB\setminus K$-indiscernible.
\item $I$ and $J$ are called colinear if there is an indiscernible sequence containing both $I$ and $J$ as increasing subsequences.
\end{enumerate}\end{definition}

\justify
Colinearity and $\approx $ are reflexive, symmetric and type-definable for a given order type. 

\justify
If $I$ and $J$ are colinear then $I\approx J$. Conversely if $I\approx J$ then they are in the transitive closure of collinerarity (which is also the transitive and symmetric closure of $l$). We will show that $\approx$ is an equivalence relation, and so that it is the transitive closure of collinerarity, which is going to be an important result. Now let $q(x_{i}$ : $i<\omega )$ be the type of an indiscernible sequence indexed by $\omega$.

\begin{remark}\label{hypcol1}\begin{enumerate}
\item[(1)] Assume $I\approx J$ for $I,J \models q $, then by compactness there is $I',J',L$ such that $Q(I,I',J,J',L)$ and $I',J',L \models q$. Moreover for any linearly ordered set $\delta \supseteq \omega$ by compactness we can find indiscenibles sequences indexed by $\delta $, $I_{0}$ and $J_{0}$ such that $I\subseteq I_{0}, J\subseteq J_{0}$ and  $I'_{0},I'_{0},L_{0}\equiv J_{0}$ such that $Q(I_{0},I_{0}',J_{0},J_{0}',L_{0})$, so $I_{0}\approx J_{0}$.
\item[(2)] In the proof of the following theorem, we will use indiscenible sequences indexed by $\mathbb{Z}$, for the following reason : since the order of $\mathbb{Z}$ has no end points we can use finite satisfiability on "both sides". More concretely : 
Let $I^{0}, I^{1}$ be $\mathbb{Z}-$indexed indiscernibles sequences with $l(I^{0},I^{1})$, then $I^{i}$ is a total Kim-Morley sequence over $I_{j}$ for $i \neq j \in \lbrace 0,1 \rbrace $. If in addition $I^{i}$ is $I^{j}B$-indiscernible then $I^{i}\forkindep^{K}_{I^{j}}B$  follows.
\end{enumerate}\end{remark}
\justify
Proof of (2) : $I^{0}$ is total Kim-Morley over $I^{1}$ : For $k <\omega $ we have that $I^{0}_{\geq k}\forkindep^{\mathcal{U}}_{I^{1}}I^{0}_{< k}$, and so $I^{0}_{\geq k}\forkindep^{K}_{I^{1}}I^{0}_{< k}$. We can see that this only uses the fact that the order type of $I^{1} $ has no minimal element. $I^{1}$ is total Kim-Morley over $I^{0}$ : For $k <\omega $ we have that $I^{1}_{< k}\forkindep^{\mathcal{U}}_{I^{0}}I^{1}_{\geq k}$ so $I^{1}_{< k}\forkindep^{K}_{I^{0}}I^{1}_{\geq k}$ and by symmetry $I^{1}_{\geq k}\forkindep^{K}_{I^{0}}I^{1}_{< k}$ (this only uses the fact that the order type of $I^{1} $ has no greater element). The last point is then just \cref{chaintotal3}.
 
\justify
We will also make use of the following fact, which is true in any G-compact theory : 

\begin{fact}\label{hypcol3} Let $I$ be an indiscernible sequence with no maximal element. Let $J$ and $K$ be sequences of the same order type. If $l(I,J)$ and $l(I,K)$ we have that $J\equiv_{I}^{L}K$.\end{fact}

\justify
\begin{proof}
Since $I$ has no maximal element by compactness we can find $I'$ of the same order type as $J$ and $K$ such that $I \frown I' \frown J$ and $I \frown I' \frown K$ are indiscernible. By extending these two sequences by omega times the order type of $J$ and $K$ we get two indiscernible sequences of sequences that begins with $(I',J)$ and $(I',k)$, so $J \equiv^{L}_{I} I'\equiv^{L}_{I} K$.\end{proof}

\justify
This result is \cite[Theorem 3.3]{kim2021weak}, the proof works similarly as the original using the previous results. It is the key result of Kim's article and we will apply it in the next subsection.

\begin{theorem}\label{hypcolth}$\approx$ is transitive in $q$, so in an $NSOP_{1}$ theory with existence for hyperimaginaries $\approx$ is the transitive closure of colinearity for indiscernible sequences of hyperimaginaries of the same order type.
\end{theorem}

\subsection{Weak Canonical Bases in NSOP1 Theories}

\begin{definition} Assume that $p(x)= \tp(c/b)$ is a strong Lascar type, meaning that $c'\equiv_{b}^{L} c$ if $c' \models p$.
\begin{enumerate}
\item[(1)] Let $\mathcal{B}= \lbrace a \in dcl^{heq}(b)$ : $ c \forkindep^{K}_{a}b$ and $\tp(c/a)$ is a Lascar strong type$ \rbrace $

A canonical base of $p$, written $\cb(p)$ or $\cb(c/b)$, is the smallest element (up to definable closures) in $\mathcal{B}$, that is : 
$\cb(p)\in \mathcal{B}$ and $\cb(p)\in dcl^{heq}(a)$ for any $a \in \mathcal{B}$. We say that $T$ has
canonical bases if so have all Lascar strong types in $T$. $\cb(p)$ is unique (if it exists) up to interdefinability.

\item[(2)] We additionally assume that $b$ is boundedly closed. Let $\mathcal{WB}= \lbrace a \in dcl^{heq}(b)$ : $ c \forkindep^{K}_{a}b$ and $a$ is boundedly closed$ \rbrace $. By the weak canonical base of $p$, written $\wcb(p)$ or $\wcb(c/b)$, we mean the smallest element in $\mathcal{WB}$ (if it exists), that is : $\wcb(p) \in \mathcal{WB}$ and $\wcb(p) \in dcl^{heq}(a)$ for any $a \in \mathcal{WB} $. When $b$ is not boundedly closed then $\wcb(c/b)$ means $\wcb(c/bdd(b))$. We say that $T$ has weak canonical bases if $\wcb(c/b)$ exist for all $c,b \in \mathbb{M}^{heq}$.
\end{enumerate}\end{definition}

\justify
If $\cb(p)$ exists, then so does $\wcb(p)$ and we have that $\wcb(p) =bdd(\cb(p))$.

\justify
Now let $I = ( c_{i}$ : $i<\omega )$ be a total Kim-Morley sequence over $b$ with $c=c_{0}$. By compactness there is $J = ( c_{\omega + i}$ : $i<\omega ) $ such that $I \frown J$ is a total Kim-Morley sequence over $b$. We also fix $e \in \mathbb{M}^{heq}$ which is $I/ \approx$ $ = J / \approx $, which makes sense by \cref{hypcolth}. Then $ e \in dcl^{heq}(I) \cap dcl^{heq}(J)$. We show that $bdd(e)= \wcb(I/b)$.

\begin{prop}\label{wcb2} Suppose that $I \forkindep^{K}_{a}b$ for some boundedly closed $a$ such that $a \in dcl^{heq}(b)$.
\begin{enumerate}
\item[(1)] If $I_{0}\equiv_{a}I$ then $I_{0}\approx I$.
\item[(2)] $e\in dcl^{heq}(a)$ (so $bdd(e)\in dcl^{heq}(a)$ since $a$ is boundedly closed).
\end{enumerate}\end{prop}

\justify
\begin{proof} (1) By \cref{hypcol1} (2) we have that $I \forkindep^{K}_{b}J$, so by transitivity $I \forkindep^{K}_{a}J$. From this, $a$-indiscernibility and finite character we deduce that $I\frown J$ is total Kim-Morley over $a$. $I_{0}\equiv_{a}I$, so there is $J_{0}$ such that  $I_{0}J_{0}\equiv_{a}IJ$. By extension there is $I'J'\equiv_{a}IJ$ such that $I'J' \forkindep^{K}_{a}IJI_{0}J_{0}$.
\justify
$a$ is boundedly closed so $J_{0}\equiv_{a}^{L}J'\equiv_{a}^{L}J$, and we have the following :\begin{center}
 $\left\{ \begin{array}{l}
        \ J'\equiv_{a}^{L}J$, $J' \forkindep^{K}_{a}I'$, $J \forkindep^{K}_{a}I$, $I' \forkindep^{K}_{a}I   \\
        \ J'\equiv_{a}^{L}J_{0}$, $J' \forkindep^{K}_{a}I'$, $J_{0} \forkindep^{K}_{a}I_{0}$, $I' \forkindep^{K}_{a}I_{0} 
    \end{array}
    \right\} $ 
\end{center}
    
\justify By amalgamation there is $J_{1},J_{2}$ such that $\left\{ \begin{array}{l}
        \ J_{1}\models \tp(J'/aI')\cup \tp(J/aI) \\
        \ J_{2}\models \tp(J'/aI')\cup \tp(J_{0}/aI_{0})
    \end{array}
    \right. $, so $\left\{ \begin{array}{l}
        \ l(I,J_{1}),l(I',J_{1}) \\
        \ l(I',J_{2}),l(I_{0},J_{2})
    \end{array}
    \right. $, and $I,I_{0}$ are in the transitive closure of colinearity. By \cref{hypcolth} $I_{0}\approx I$.
    
    \justify
(2) Let $f$ be an $a$-automorphism and $I_{0}:=f(I)$. Since $I_{0}\equiv_{a}I$ by (1) $I_{0}\approx I$, and $e$ is fixed by $f$, so $e \in dcl^{heq}(a)$.
\end{proof}
 
\begin{prop}\label{wcb3} $J \forkindep^{K}_{I}b$,  $J \forkindep^{K}_{e}I$,  $I \forkindep^{K}_{e}b$ and  $c \forkindep^{K}_{e}b$, so $I$ is total Kim-Morley over $e$. \end{prop}

\justify
\begin{proof} $J$ is total Kim-Morley over $I$ : $J_{<i} \forkindep^{\mathcal{U}}_{I}J_{\geq i}$, so $J_{\geq i} \forkindep^{K}_{I}J_{<i}$. By \cref{chaintotal3}, because $J$ is $Ib$-indiscernible, we have that $J \forkindep^{K}_{I}b$. We have $e\in dcl^{heq}(b)$ by \cref{wcb2} and $c\equiv_{b}c'$ for every $c'$ element of $J$.

\justify
\textbf{Claim :} There are $U,L$ such that $U\equiv_{e}I$, $U \forkindep^{K}_{e}I$ and $l(I,L),l(U,L)$. 

\justify
Proof of the Claim : By compactness we find a negative part for $I$ : there is an $e$-indiscernible sequence $I_{0}(\supseteq I)$ indexed by $\mathbb{Z}$. By extension there is $U_{0}\equiv_{e}I_{0}$ such that $U_{0} \forkindep^{K}_{e}I_{0}$. Then $U^{\geq 0}_{0}\equiv_{e}I$ and $e = I/\approx $, so $U^{\geq 0}_{0}\approx I$. Now $I_{0}$ and $I$ are colinear, and also $U_{0}$ and $U^{\geq 0}_{0}$, so $I_{0} \approx U_{0}$.

\justify
We can find a sufficiently large $e$-indiscernible sequence $\textbf{U}_{0}$ containing $U_{0}$ such that $\textbf{U}_{0} \forkindep^{K}_{e}I_{0}$. By applying the same method as in Claim 1 of \cite[Theorem 3.3]{kim2021weak} (and reducing $\textbf{U}_{0}$) we can assume that we have $I',J',K\equiv U_{0}$ such that $Q_{0}(I_{0},I',U_{0},J',K)$ and  $I_{0} \forkindep^{K}_{e}U_{0}$, $U_{0}\equiv_{e}I_{0}$.

\justify
Then $I_{0},I',U_{0},J',K$ are $\mathbb{Z}$-indexed sequences of indiscernibles, so we can apply \cref{hypcol1} to them, giving us $U_{0}\forkindep^{K}_{J'}I_{0}I$ and $K\forkindep^{K}_{J'}U_{0}I_{0}I'$, so $K$, $U_{0}$ and $I_{0}I'$ are pairwise Kim-independent over $J'$.
    
\justify
By \cref{2a2indeph} there is $U'_{0}$ such that $U'_{0}\equiv^{L}_{J'I_{0}I'} U_{0}$, $U'_{0}\forkindep^{K}_{J'I_{0}I'}K$, $U'_{0}\equiv^{L}_{J'K} U_{0}$, $U'_{0}\forkindep^{K}_{J'K}I_{0}I'$. $e\in dcl^{heq}(I)$ so  $I_{0} \forkindep^{K}_{e}U'_{0}$ and $U'_{0}\equiv_{e}I_{0}$. Also $U'_{0}\forkindep^{K}_{K}J' $ so by transitivity $U'_{0}\forkindep^{K}_{K}I_{0} $. Since $l(J',K),l(I',K)$ by \cref{hypcol1} (and since these are $\mathbb{Z}$-indexed indiscernible) $J'\equiv^{L}_{K}I'$.
    
 \justify
We have $J'\equiv^{L}_{K} I'$, $U'_{0}\forkindep^{K}_{K}I_{0}$, $U'_{0}\forkindep^{K}_{K}J'$, $I'\forkindep^{K}_{K}I_{0}$. By type amalgamation there is $L \models \tp(J'/U'_{0}K)\cup \tp(I'/I_{0}K)$ such that $L\forkindep^{K}_{K}U'_{0}I_{0}$. Now $l(U'_{0},L),l(I_{0},L)$ and by taking $U=U'^{\geq 0}_{0}$ $U\equiv_{e}I$, $U \forkindep^{K}_{e}I$ and $l(I,L),l(U,L)$, which proves the Claim.

\justify
By compactness we can lengthen $L$ (and keep  $l(I,L),l(U,L)$) and apply Erdös-Rado, so we can assume that $L$ is $IU$-indiscernible. $L$ is total Kim-Morley over $U$ by finite satisfiability, so by \cref{hypcol1} $L \forkindep^{K}_{U}I$, and since $U \equiv_{e} I$ and $e \in dcl^{heq}(I)$, we have that $e \in dcl^{heq}(U) $ and $L \forkindep^{K}_{Ue}I$ so by transitivity, since $U \forkindep^{K}_{e}I$, $L \forkindep^{K}_{e}I$ is satisfied. We have $l(I,L)$ and $l(I,J)$, so $L \equiv^{L}_{I} J$ as sequences indexed on $\omega $, so $L \equiv_{Ie} J$ and $L \forkindep^{K}_{e}I$ whence $J \forkindep^{K}_{e}I$.

\justify
The rest follows : we have $J \forkindep^{K}_{Ie}b$, so by transitivity $J \forkindep^{K}_{e}Ib$.

\justify
We have that $I_{\geq i} \forkindep^{\mathcal{U}}_{J}I_{< i}$ for all $i<\omega $, so $I$ is total Kim-Morley over $J$ and is $Jb$-indiscernible. By \cref{chaintotal3} $I \forkindep^{K}_{J}b$, and by transitivity (with $I \forkindep^{K}_{e}J$) $I \forkindep^{K}_{e}b$. Then, clearly $c\forkindep^{K}_{e}b$.\end{proof}

\begin{theorem}\label{4.4wcb} The weak canonical base of the type of a total Kim-Morley sequence over a boundedly closed hyperimaginary exists. Namely, $\wcb(I/b) = bdd(e)$. Moreover if $\wcb(p)$ exists then $\wcb(p) \in bdd(e)$.
\end{theorem}

\justify
\begin{proof} $I \forkindep^{K}_{b}b$ so by \cref{wcb2} $e \in dcl^{heq}(b)$. By \cref{wcb3} $I \forkindep^{K}_{e}b$, so $I \forkindep^{K}_{bdd(e)}J$. If $a$ is boudedly closed such that $a\in dcl^{heq}(b)$ and $I \forkindep^{K}_{a}b$, by \cref{wcb2} $bdd(e)\in dcl^{heq}(a)$, so $\wcb(I/b)=bdd(e)$. For the second point, $c\forkindep^{K}_{bdd(e)}b$, so if $\wcb(p)$ exists then $\wcb(p) \in bdd(e)$.\end{proof}

\begin{remark}\label{remarkwcb} Let $p(x)=\tp(c/b)$, we define : \begin{center}
$\textbf{e}= \bigcap  \lbrace dcl^{heq}(\wcb(L/b))$ : $L$ is a total Kim-Morley sequence in $p \rbrace $.
\end{center}
\begin{enumerate}
\item[(1)] Assume that $g=\wcb(p)$ exists.
\begin{enumerate}

\item[(a)] $c \forkindep^{K}_{g}b$ holds and by \cref{2.7wcb} (and \cref{doublemorley}) there is $L_{0}= ( c_{i}$ : $i<\omega)$ with $c_{0}=c$ such that $L_{0}\forkindep^{K}_{g}b$ and $L_{0}$ is a total Kim-Morley sequence over both $g$ and $b$. By \cref{4.4wcb} $e_{0}:= \wcb(L_{0}/b)$ exists and $L_{0}\forkindep^{K}_{e_{0}}b$ so $c\forkindep^{K}_{e_{0}}b$ and $g \in  dcl^{heq}(e_{0})$. But since $L_{0}\forkindep^{K}_{g}b$ we have $ e_{0} \in dcl^{heq}(g)$ and $g \sim e_{0}$.

\item[(b)] Moreover we have $g \sim \textbf{e}$ : Due to (a) we have $\textbf{e}\in dcl^{heq}(g) $. Let $e_{L}=\wcb(L/b)$, we have $L \forkindep^{K}_{e_{L}}b$ so $c \forkindep^{K}_{e_{L}}b$, hence $g \in dcl^{heq}(e_{L})$, so $g \in dcl^{heq}(\textbf{e})$ and $g \sim \textbf{e}$.
\end{enumerate}

\item[(2)] Conversely if $c\forkindep^{K}_{\textbf{e}}b$ then $\textbf{e} = \wcb(p)$ : We have $\textbf{e} \sim bdd(\textbf{e}) \in dcl^{heq}(b)$.

Now assume that $c\forkindep^{K}_{a}b$ with $a \sim bdd(a)\in dcl^{heq}(b)$. Then by \cref{2.7wcb} and \cref{doublemorley} there is a total Kim-Morley sequence $L$ over both $a$ and $b$ such that $L\forkindep^{K}_{a}b$. Hence for $e_{L}=\wcb(L/b)$, we have that $\textbf{e} \in dcl^{heq}(e_{L}) \in dcl^{heq}(a)$.
\end{enumerate}\end{remark}

\begin{prop}\label{hypwcb1} The following are equivalent :

\begin{enumerate}
\item[(1)] For $p(x)=\tp(c/b)$, $\wcb(p)$ exists.
\item[(2)] There is a total Kim-Morley sequence $L = ( c_{i},i<\omega)$ over $b$ with $c_{0} = c$ such that for any $a \in dcl^{heq}(b)$ with  $c\forkindep^{K}_{a}b$, some $L'\equiv_{c}L$ with $L'\approx L $ is total Kim-Morley over $a$.
\end{enumerate}\end{prop}

\justify
\begin{proof} $(1)\Rightarrow (2)$ : Let $g=\wcb(p)$. As in \cref{remarkwcb} (1), $g = e_{L}:= \wcb(L/b)$ where $L$ is a total Kim-Morley sequence over $g$ and $b$ such that $L\forkindep^{K}_{g}b$. Now if $c\forkindep^{K}_{a}b$, we have $g \in bdd(a)$ and $c\forkindep^{K}_{g}bdd(a)$. So by \cref{2.7wcb} there is $L'\equiv_{cg}L$ such that $L'$ is total Kim-Morley over both $a$ and $g$, then $g\sim e_{L}$ and $L'\equiv_{e_{L}}L$, so $L'\approx L$.

\justify
$(2)\Rightarrow (1)$ :  Let $g=\wcb(L/b)$ with $L$ a total Kim-Morley sequence in $p$ satisfying $(2)$. So $g \in dcl^{heq}(b)$ and $c\forkindep^{K}_{g}b$. Let $a \sim  bdd(a) \in dcl^{heq}(b)$ such that $c\forkindep^{K}_{a}b$, so by hypothesis there is $L'\equiv_{c}L$ such that $L'\approx L $  and $L'$ is total Kim-Morley over $a$. Let $g'=\wcb(L'/a) \in dcl^{heq}(a)$. Now $L'/\approx$ $ = L/\approx $ is interbounded with $g'$, so $g$ and $g'$ are interbounded and both boundedly closed, so $g' \sim g \in dcl^{heq}(a)$ and $g = \wcb(p)$.\end{proof}

\justify
The following proposition is a translation of \cite[Proposition 4.7]{kim2021weak} in the context of hyperimaginaries.

\begin{prop}\label{4.7wcb} We assume that T has weak canonical bases. Let $c\in \mathbb{M}^{heq}$ be interdefinable with $C$ a set of hyperimaginaries. Then :
\begin{enumerate}
\item[(1)] $\wcb(c/b) = bdd(\wcb(c'/b)$ : $c'$ finite tuples from $C$)
\item[(2)] For some $a$ boundedly closed in $dcl^{heq}(c)$, if $\wcb(b/ac') \in dcl^{heq}(a)$ for every finite tuple $c' \in C$, then $\wcb(b/c) \in dcl^{heq}(a)$.
\end{enumerate}\end{prop}

\justify
\begin{proof} (1) : Let $U:= \wcb(c/b)(\in dcl^{heq}(b))$, and let $V= bdd(\wcb(c'/b)$ : $c'$ finite tuples from $C$).

\justify
Let $(c^{j}$ : $j \in F)$ be the collection of all finite tuples from $C$, and let $W=(w^{j}$ : $j \in F)$ where $w^{j}=\wcb(c^{j}/b)$. Hence $V=bdd(W)$. $c\forkindep^{K}_{U}b$, so $c^{j}\forkindep^{K}_{U}b$ for all $j\in F$ and $w^{j} \in dcl^{heq}(U)$, so $V \in dcl^{heq}(U)$.

\justify
To show that $U \in dcl^{heq}(V)$ it is enough to show that $c\forkindep^{K}_{W}b$ (since then $c\forkindep^{K}_{bdd(W)}b$). By local character and symmetry is is enough to show that for every $k\in F$ $\tp(b/c^{k}w^{k})$ does not Kim-fork over $W$. We have that $c^{k}\forkindep^{K}_{w^{k}}b$, so $c^{k}\forkindep^{K}_{w^{k}}W$. By applying \cref{doublemorley} there is a total Kim-Morley sequence $(c^{i}_{k}$ : $i<\omega )$ over $w^{k}$ that is also Kim-Morley over $W$. Then since $c^{k}\forkindep^{K}_{w^{k}}b$ by \cref{chaincondhyp} there is $b'\equiv_{w^{k}c^{k}_{0}}b$ such that $(c^{i}_{k}$ : $i<\omega)$ is $b'w^{k}$-indiscernible. Let $p = \tp(b'w^{k}c^{k}_{0})=\tp(bw^{k}c^{k}_{0})$, $\lbrace p(x,w^{k}c^{k}_{i})$ : $ i< \omega \rbrace$ is consistent because $b'$ satisfies it, so by \cref{witnessinghyp2} $p$ does not Kim-fork over $W$.

\justify
(2) : Let $(c_{j}$, $j\in F)$ be as in (1). Let $R=(r^{j}$, $j\in P)\in dcl^{heq}(a)$ with $r^{j}=\wcb(b/ac^{j})$. It is enough to show that $b\forkindep^{K}_{R}c$ since then $\wcb(b/c) \in dcl^{heq}(R)$, and by hypothesis $R \in dcl^{heq}(a)$. As in (1) by local character it is enough to show that $p(x,r^{k}b) = \tp(c^{k}/r^{k}b)$ does not Kim-fork over $R$ for any $k\in F$. 

\justify
We have that $b\forkindep^{K}_{r^{k}}ac^{k} $, so $b\forkindep^{K}_{r^{k}}R $. By \cref{doublemorley} there is a sequence $(b_{j}$ : $j< \omega)$ with $b_{0}=b$ that is total Kim-Morley over $r^{k}$ and Kim-Morley over $R$. Then since $b\forkindep^{K}_{r^{k}}ac^{k}$ by \cref{chaincondhyp} there is $c'^{k}\equiv_{r^{k}b}c^{k}$ such that $(b_{j}$ : $i<\omega)$ is $c'^{k}w^{k}$-indiscernible, so in particular $c'^{k}\models \lbrace p(x,r^{k},b_{i})$ : $i<\omega\rbrace $, and we conclude by \cref{witnessinghyp2}. \end{proof}

\bibliographystyle{plain}
\bibliography{Ref.bib}

\end{document}